\documentclass{article}
\usepackage[utf8x]{inputenc}
\usepackage{amsmath, amssymb, amsthm, mathrsfs,cleveref}
\usepackage{bbold}
\usepackage{tikz-cd}
\usepackage[margin=3cm]{geometry}

\newtheorem{thm}{Theorem}[section] \newtheorem*{theorem*}{Theorem}
\newtheorem{lemma}[thm]{Lemma} \newtheorem*{lemma*}{Lemma}
\newtheorem{corollary}[thm]{Corollary} \newtheorem*{corollary*}{Corollary}
\newtheorem{prop}[thm]{Proposition} \newtheorem*{proposition*}{Proposition}

\newtheorem{definition}[thm]{Definition}
\newtheorem*{definition*}{Definition}
\newtheorem{remark}[thm]{Remark}

\def\subjclass#1{\par\medskip
\noindent\textbf{Mathematics Subject Classification (2010):} #1}
\def\keywords#1{\par\medskip
\noindent\textbf{Keywords.} #1}

\newcommand{\bJ}{{\mathbb J}}

\newcommand{\N}{{\mathbb N}}
\newcommand{\Q}{{\mathbb Q}}
\newcommand{\R}{{\mathbb R}}
\newcommand{\T}{{\mathbb T}}
\newcommand{\Z}{{\mathbb Z}}

\newcommand\cA{{\mathcal A}}

\newcommand\bA{{\mathbb A}}
\newcommand\bB{{\mathbb B}}

\newcommand\bD{{\mathbb D}}

\newcommand{\tb}{|\!|\!|}
\newcommand{\ev}{\xi}

\newcommand{\fs}{\mathfrak s}
\newcommand{\ft}{\mathfrak t}

\newcommand\eps{\epsilon}

\def\eps{\varepsilon}

\newcommand{\Span}{\operatorname{span}}

\title{Interval Translation Maps with Weakly Mixing Attractors}
\author{Henk Bruin\thanks{Faculty of Mathematics, University of Vienna, 
		Oskar Morgensternplatz 1, 1090 Vienna, Austria; {\it henk.bruin@univie.ac.at, \quad silvia.radinger@univie.ac.at}.
		We gratefully acknowledge the support of the Austrian Exchange Service (WTZ Project PL 15/2022). We also thank Reem Yassawi for the fruitful conversations.}, \quad Silvia Radinger\footnotemark[1]}

\date{\today}

\begin{document}

	\maketitle
	
	\begin{abstract}
		We study linear recurrence and weak mixing of a two-parameter family of interval translation maps
		$T_{\alpha,\beta}$ for the subset of parameter space where $T_{\alpha,\beta}$ has a Cantor attractor.
		For this class, there is a procedure similar to the Rauzy induction which acts as a dynamical system
		$G$ on parameter space, which was used previously to decide whether $T_{\alpha,\beta}$
		has an attracting Cantor set, and if so, whether $T_{\alpha,\beta}$ is uniquely ergodic.
		In this paper we use properties of $G$ to decide
		whether $T_{\alpha,\beta}$ is linearly recurrent or weak mixing.
	\end{abstract}


	\section{Introduction}

In \cite{BT} the following family of interval translation maps (ITMs) was introduced:
% $$
% T_{\alpha,\beta}(x) = \begin{cases}
%          x+\alpha, & x \in [0,1-\alpha), \\
%          x+\beta, & x \in [1-\alpha,1-\beta),\\
%          x-1+\beta, \qquad & x \in [1-\beta,1].
%        \end{cases}
% $$
%

\begin{figure}[h]
\begin{center}
\begin{tikzpicture}[scale=0.5]
\node at (2.5,2.5) {$T_{\alpha,\beta}(x) = \begin{cases}
         x+\alpha, & x \in [0,1-\alpha), \\
         x+\beta, & x \in [1-\alpha,1-\beta),\\
         x-1+\beta, \qquad & x \in [1-\beta,1].
       \end{cases}
$};
\draw[-] (15,0) -- (21,0) -- (21,6) -- (15,6) -- (15,0);
\draw[-,dashed] (15,0) -- (21,6);
\draw[-] (20,-0.2) -- (20,0); \draw[-] (18.2,-0.2) -- (18.2,0);
   \node at (14.8,-0.8) {\small $0$}; \node at (21.2,-0.8) {\small $1$};
    \node at (17.6,-0.8) {\small $1-\alpha$}; \node at (19.8,-0.8) {\small $1-\beta$};
     \node at (14.5,3) {\small $\alpha$}; \node at (21.5,1) {\small $\beta$};
 \draw[-, thick] (15,2.8) -- (18.2,6);
  \draw[-, thick] (18.2,4.2) -- (20,6);
   \draw[-, thick] (20,0) -- (21,1);
 \end{tikzpicture}
% \caption{}\label{fig:SFT}
\end{center}
\end{figure}

\vspace{-5mm}
\noindent
on the parameter space $U = \{ (\alpha,\beta) : 0 < \beta \leq \alpha \leq 1\}$.
Viewed as circle map, $T_{\alpha,\beta}$ is a piecewise rotation on two pieces; it was studied in greater generality in \cite{AFHS21, B06}.

ITMs come in: (i) finite type if reduced to a set of subintervals, it is an interval exchange transformation (IET),
and (ii) if the only compact invariant set are  Cantor sets,
and (iii) mixtures of the two.
Determining the type for this family goes via  renormalization consisting of the first return map to $[1-\alpha,1]$ and rescaling to unit size,
in analogy to Rauzy induction for IETs.
This transforms $T_{\alpha,\beta}$ into $T_{\alpha',\beta'}$ where
\begin{equation}\label{eq:G}
(\alpha', \beta') = G(\alpha,\beta) = \left(\frac{\beta}{\alpha}\ , \ \frac{\beta-1}{\alpha} + \Big\lfloor\frac{1}{\alpha}\Big\rfloor\right).
\end{equation}
Note that $G(U) = U \cup L$ for $L = \{ (\alpha,\beta) : \alpha - 1 \leq \beta \leq 0 \leq \alpha \leq 1\}$
and exactly the parameters in
$\Omega_\infty := \{ (\alpha,\beta) : G^n(\alpha,\beta) \in U^\circ \text{ for all } n \geq 0\}$
have the property that $T_{\alpha, \beta}$ is of infinite type, i.e.,
$\Omega := \bigcap_{n \geq 0} \overline{ T_{\alpha,\beta}^n([0,1]) }$ is a Cantor set on which $T_{\alpha,\beta}$ acts as a minimal endomorphism. See Figure~\ref{fig2} for the parameter set associated to type infinity ITMs.

\begin{figure}[ht]
\setlength{\unitlength}{1cm}
 \begin{picture}(10,6)
 \put(3.5,-4.5){\resizebox{8 cm}{16 cm}{\includegraphics{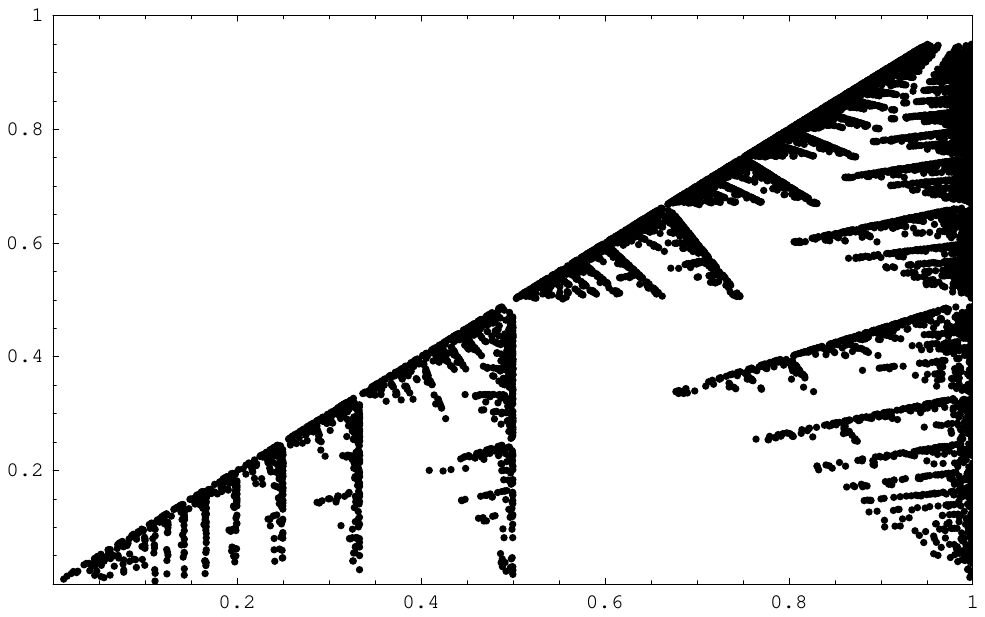}}}
 \end{picture}
\caption{Approximation of the set $\Omega_\infty$ of type infinity parameters with $10,000$ pixels.
}
\label{fig2}
\end{figure}

Symbolically, $T_{\alpha,\beta}$ is described by an $S$-adic subshift $(X,\sigma)$, based on a sequence of
substitutions $\chi_{k_i}$, $k_i \in \N = \{1,2,3,\dots\}$,
and we call $T_{\alpha,\beta}$ linearly recurrent if
the subshift $(X,\sigma)$ is, i.e., there is $L$ such that for every $x \in X$, every subword $w$
reappears in $x$ with gap $\leq L|w|$.
Theorem~\ref{thm:LR} gives a precise condition on the sequence $(k_i)_{i \in \N}$
that results in linear recurrence.

Unique ergodicity of $T_{\alpha,\beta}$ fails if the sequence $(k_i)$  increases exponentially fast, see \cite[Theorem 11]{BT}.
In this non-generic situation, there are two ergodic measures.
We can also conclude from its representation as S-adic shifts, that
$T_{\alpha,\beta}$ is never strongly mixing, see e.g.\ \cite[Theorem 6.79]{B23}.

Instead, the other property we investigate in this paper
 is weakly mixing of $T_{\alpha,\beta}$.
 %w.r.t.\ its unique invariant measure.
 General results in weak mixing for interval exchange transformations were obtained by Nogueira \& Rudolph \cite{NR97} (generic IETs),
Sina\u{\i} \& Ulcigrai \cite{SU05} (``periodic'' IETs)
and Avila \& Forni \cite{AF07} (Lebesgue typical IETs).
Recent extensions of these results can be found in e.g.\ \cite{AD16, AFS,  AL18}.
However, the relative simplicity of the family $\{ T_{\alpha,\beta}\}$ allows us to use methods from linear algebra (rather than results of Veech and Teichm\"uller theory)
combined with a general condition on Bratteli-Vershik systems for existence of eigenvalues of the Koopman operator.
This condition goes back to Host \cite{Host86}, and worked out in more generals settings in \cite{BCY, BDM05,BDM10,CDHM03,DFM19}.
 We prove that all ``periodic'' ITMs in our class (i.e., those for which the  corresponding sequence $(k_i)_{i\in\N}$ is (pre-)periodic)
as well as ``typical''  (in a sense made precise in Theorem~\ref{thm:wm_liminfk}) ITMs in our class are weakly mixing.

The paper is structured as follows. In Section~\ref{sec:LR} we characterize the linear recurrent ITMs by giving an if and only if condition on the index sequence $(k_i)_{i\in \N}$ of the substitutions in the S-adic representation.
Section~\ref{sec:WM} gives the results on weak mixing. When the S-adic representation is viewed as non-autonomous sequences of toral automorphisms $(A_k)_{k \geq 1}$, the condition
for the existence of an eigenvalue $e^{2\pi i \ev}$ of the Koopman operator is close (although not equivalent) to
the vector $\vec \ev := (\ev,\ev,\ev)\in \T^3$ belonging to the stable space $W^s(\vec 0)$ of $(A_k)_{k \geq 1}$.
In Section~\ref{sec:wm-Lyap} we show that $W^s(\vec 0)$ is one-dimensional,
so that the absence of eigenvalues becomes the generic situation.
Section~\ref{sec:wm-periodic} gives algebraic reasons why ITMs with  (pre-)periodic sequences $(k_i)_{i \in \N}$ are weak mixing.
Section~\ref{sec:wm-stable}  investigates the stable direction further, showing that it is uniquely determined by $(k_i)_{i \in \N}$.
In Section~\ref{sec:wm-continuous} we show that $\vec \ev \in W^s(\vec 0)$ is a necessary, although not a sufficient, condition for the existence of continuous eigenvalues.
Finally, in Section~\ref{sec:wm-measurable} we add the generic condition $\liminf_i  k_i < \infty$.
Under this condition the ITM is uniquely ergodic, and the Koopman operator has
no measurable eigenvalues if $\vec \ev \not\in W^s(\vec 0)$.

\section{Linear recurrence}\label{sec:LR}

We use symbolic dynamics w.r.t.\ the partition $\{ [0, 1-\alpha), [1-\alpha,1-\beta), [1-\beta,1]\}$,
with symbols $1,2,3$, respectively.
One renormalization step is  given by the substitution
\begin{equation}\label{eq:chi}
\chi_k: \begin{cases}
         1 \to 2\\
         2 \to 31^k\\
         3 \to 31^{k-1}
        \end{cases}
        \qquad \text{ for } k = \Big\lfloor \frac{1}{\alpha} \Big\rfloor \in \N.
\end{equation}
The associate matrix and its inverse are
$$
A_k = \begin{pmatrix}  0 & k & k-1 \\ 1 & 0 & 0 \\ 0 & 1 & 1    \end{pmatrix}
\qquad \text{ and }\qquad
A_k^{-1} =  \begin{pmatrix}  0 & 1 & 0 \\ 1 & 0 & 1-k \\ -1 & 0 & k  \end{pmatrix}.
$$
Note that $\det(A_k) = \det(A_k^{-1}) = -1$ and the characteristic polynomial of $A_k$ 
is $P_k(\lambda) = \lambda^3-\lambda^2-k\lambda+1$.

Every ITM of infinite type in this family is uniquely characterised by
a sequence $(k_i)_{i \in \N} \subset \N$ such that
\begin{equation}\label{eq:k2}
k_{2i} > 1 \text{ for infinitely many } i \in \N \text{ and } k_{2i-1} > 1 \text{ for infinitely many } i \in \N.
\end{equation}
The itinerary of the point $1 \in \left[0,1\right]$ is then
$$
\rho =\lim_{i\to\infty} \chi_{k_1} \circ \chi_{k_2} \circ \chi_{k_3} \circ \cdots \circ \chi_{k_i}(3).
$$
Let $X$ be the closure of $\{ \sigma^n(\rho) \}_{n \in \N}$ where $\sigma$ is the left-shift.
For such sequences, the attractor $\Omega$ of the ITM is a Cantor set, on which the action 
is isomorphic to an $S$-adic shift $(X,\sigma)$ with
associated matrices $A_{k_i}$.
In \cite{BT} conditions are given under which $(\Omega,T_{\alpha,\beta})$ is uniquely ergodic, see also Corollary~\ref{cor:EU}.
For the first result on linear recurrence for our family of ITMs, we need certain properties for the $S$-adic shift.

A $S$-adic subshift based on substitutions $\chi_i: \mathcal{A}_i \to \mathcal{A}^+_{i-1}$ is called {\em primitive} if for all $m\in\N$ there exists $n\geq m$ such that for all $a\in\mathcal{A}_n$
$$
	\chi_{m} \circ \cdots\circ \chi_n (a) \text{ contains every }
	b \in \mathcal{A}_{m-1}.
$$
The corresponding shift space $X$ is the shift orbit closure of the set of accumulation points of $\{ \chi_1 \circ \cdots\circ \chi_n (a) : n \in \N, a \in {\mathcal A}_n\}$.
A subshift is {\em aperiodic} if there exists no $x\in X$ such that $\sigma^k(x) = x$ for some $k\in\N$.

A substitution is called {\em left-proper}, if for all letters $a\in\mathcal{A}$ the word $\chi(a)$ has the same starting letter. While the substitutions $\chi_{k_i}$ based on an ITM of infinite type are not left-proper, any telescoping $\chi_k \circ \chi_{k_{i+1}}$ is, because
$$
\chi_{k_i} \circ \chi_{k_{i+1}}: \begin{cases}
	1 \to 31^{k_i} \\
	2 \to 31^{k_i-1}2^{k_{i+1}}\\
	3 \to 31^{k_i-1}2^{k_{i+1}-1}.
\end{cases}
$$
Thus $\rho$, the itinerary of the point $1 \in \left[0,1\right]$, is the unique one-sided fixed point under $(\chi_{k_{i}})_{i\in \N}$.

Next we need to show that our $S$-adic shift $(X,\sigma)$ is aperiodic.
In the literature one can find some results in this direction,
e.g., \cite[Lemma 3.3]{BCDLP} requires that the substitutions are unimodular, primitive and proper, where we only have left-proper.
The next lemma uses a weak form a recognizability,
which does hold in our case.

\begin{definition}\label{def:recog}
Let $(X, \sigma)$ be a subshift based on the substitution $\chi$.
We call $\chi$ {\em combinatorially recognizable}
if there is $N \in \N$ such that for every word $w$ in $X$
of length $|w| \geq N$ and every $v = v_1v_2v_3 \dots v_n$ such that $w$ appears twice in $\chi(v)$, then
the first appearance of $w$ starts at $\chi(v_1 \cdots v_i)$
if and only if the second appearance of $w$ starts at $\chi(v_1 \cdots v_j)$ for some $0 \leq i < j < n$.
(Here $i = 0$ means that $\chi(v)$ starts with $w$.)

The $S$-adic subshift $(X, \sigma)$ based on substitutions $(\chi_i)_{i\in\N}$ is recognizable if each $\chi_i$ is recognizable
(not necessarily with the same $N$).
\end{definition}

\begin{lemma}\label{lm:aperiodic}
	Let $(X, \sigma)$ be an injective, combinatorially recognizable $S$-adic subshift based on substitutions $\chi_i:\mathcal{A}_i \to \mathcal{A}^+_{i-1}$ such that for every $n \in \N$ there is $m > n$ such that
\begin{equation}\label{eq:increaselength}
 |\chi_n \circ \cdots \circ \chi_m(a)| > 1 \quad \text{ for every } a \in {\mathcal A}_m.
\end{equation}
Then $(X, \sigma)$ is aperiodic.
\end{lemma}
\begin{proof}
	Assume by contradiction that $(X, \sigma)$ has a periodic element $x = \sigma^k \circ \chi(y) \in X$ with period $p$. It follows that $x=w^\infty$ with $w$ a word of length $p$. We can pick $k$ in such a way that $w^\infty$ starts with $\chi_1(a)$ for some $a\in\mathcal{A}_1$.
	As $\chi_{k_1}$ is recognizable and injective, every new occurrence of $w$ must start with $\chi_1(a)$.
	In fact, there is a unique word $a \cdots b \in {\mathcal A}_1^+$ such that
	$\chi_1(a \cdots b) = w$ and
	$$
	\chi^{-1}_1(x) 	= \chi^{-1}_1(w)^\infty = (a\cdots b)^\infty.
	$$
	This unique $\chi^{-1}_1(x)$ is again periodic with period $p_1 \leq p$. Using \eqref{eq:increaselength} and  taking preimages $\chi^{-1}_m \circ \cdots \circ\chi^{-1}_1(x)$ decreases the period: $p_m < p_1$. Since we can repeat this argument, even when the period is reduced to $1$, we get a contradiction. Hence all elements of $(X, \sigma)$ are aperiodic.
\end{proof}

\begin{prop}\label{prop:subshift}
The $S$-adic subshift $(X,\sigma)$, based on substitutions $(\chi_{k_i})_{i\in\mathbb{N}}$ from an ITM of infinite type, is
primitive, combinatorially recognizable and aperiodic.
\end{prop}
\begin{proof}
	\textbf{Primitivity.}
	We prove that for any $i\in\N$ there exists $n_i$ such that the product of matrices associate to $\chi_{k_i}, \dots, \chi_{k_{i + n_i +1}}$ is strictly positive.
	
	We write $A_k$ for matrices with $k>1$. For odd or even length blocks of substitutions with $k_i = 1$ the product of associated matrices are 
	\begin{equation*}
		A^{2r-1} =	\begin{pmatrix}
			0 & 1 &  0\\
			1	& 0 & 0 \\
			r-1	& r & 1
		\end{pmatrix}, \ \ \ 
		A^{2r} =	\begin{pmatrix}
			1 	& 0 & 0\\
			0	& 1 & 0 \\
			r	& r & 1
		\end{pmatrix}.
	\end{equation*}
	
	Let $i \in \mathbb{N}$ be such that $k_i > 1$, as the ITM is of infinity type there exists an odd $m_i$ such that $k_{i+m_i} > 1$.
	
	In between $k_i$ and $k_{i+m_i}$ the positions of odd distance to $k_i$ are always equal to one, the even positions may take any value $k\geq 1$, i.e.,
	\begin{equation*}
		\left(k_i, 1, k_{i+2}, 1, \dots, 1, k_{i+2j}, 1, \dots, 1, k_{i+2n}, k_{i+m_i} \right).
	\end{equation*}
	Thus we can write the multiplication of matrices from $k_i$ to $k_{i+m_i}$ in the following way
	\begin{equation}\label{eq:p}
		A_{k_i}\left( \prod_{(k,r)} A^{2r-1} A_k \right) A^{2s} \ A_{k_{i+m_i}},
	\end{equation}
	where $s\geq 0$ and $(k,r)$ are the values of $k_j>1$ for $ i < j < i+m_i$ and $2r-1$ is the length of the chain of ones between the previous $k>1$ and $k_j$ with $r>0$.
	
	If the product in \eqref{eq:p} is empty, it is the identity matrix. If it contains only one pair $(k,r)$, it has the form
	\begin{equation*}
		\begin{pmatrix}
			0 & 1 &  0\\
			1	& 0 & 0 \\
			r-1	& r & 1
		\end{pmatrix}\begin{pmatrix}
			0 & k &  k-1\\
			1	& 0 & 0 \\
			0	& 1 & 1
		\end{pmatrix} = \begin{pmatrix}
			1 & 0 &  0\\
			0	& k & k-1 \\
			r	& (r-1)k + 1 & (r-1)(k-1)+1
		\end{pmatrix}.
	\end{equation*}
	Otherwise, if the product contains multiple such factors, we get
	\begin{align*}
		\prod_{(k,r)} A^{2r-1} A_k
		% &= \prod_{(k,r)}\begin{pmatrix}
			%	1 & 0 &  0\\
			%	0	& k & k-1 \\
			%	r	& (r-1)k + 1 & (r-1)(k-1)+1
			%\end{pmatrix} \\
			&= \prod_{(k,r)}\begin{pmatrix}
				1 & 0 &  0\\
				0	& \ast & \ast \\
				\ast	& \ast & \ast
			\end{pmatrix} = \begin{pmatrix}
				1 & 0 &  0\\
				\ast	& \ast & \ast \\
				\ast	& \ast & \ast
			\end{pmatrix},
		\end{align*}
		where the $\ast$ represent integers greater than zero. 
		% The $(2,1)$-entry is either equal to zero in case there is only one matrix in the product or it is greater then zero if there are multiple.
		As the identity matrix is minimal in each entry, we use it as the lowest possible bound for each entry in the matrix multiplication with $s> 0$
		\begin{align*}
			A_{k_i}\left( \prod_{(k,r)} A^{2r-1} A_k \right) A^{2s}
			%	&= 	\begin{pmatrix}
				%		0 & k_i &  k_i-1\\
				%		1	& 0 & 0 \\
				%		0	& 1 & 1
				%	\end{pmatrix} \begin{pmatrix}
				%	1 & 0 &  0\\
				%	{0, \ast} 	& \ast & \ast \\
				%	\ast	& \ast & \ast
				%\end{pmatrix} \begin{pmatrix}
				%1 & 0 &  0\\
				%0	& 1 & 0 \\
				%s	& s & 1
				%\end{pmatrix}\\
				&\geq %\begin{pmatrix}
				%	\ast & \ast &  \ast\\
				%	1	& 0 & 0 \\
				%	\ast	& \ast & \ast
				%\end{pmatrix}
				\begin{pmatrix}
					0 & k_i &  k_i-1\\
					1	& 0 & 0 \\
					0	& 1 & 1
				\end{pmatrix}\begin{pmatrix}
					1 & 0 &  0\\
					0	& 1 & 0 \\
					s	& s & 1
				\end{pmatrix}
				=\begin{pmatrix}
					\ast & \ast &  \ast\\
					1	& 0 & 0 \\
					\ast	& \ast & \ast
				\end{pmatrix}
			\end{align*}
			In case $s=0$, the result is $A_{k_i}$, which is the entry-wise lower possibility. Thus
			\begin{align*}
				A_{k_i}\left( \prod_{(k,r)} A^{2r-1} A_k \right) A^{2s} A_{k_{i+m_i}}
				%&= 	\begin{pmatrix}
				%	\ast & \ast &  \ast\\
				%	1	& 0 & 0 \\
				%	\ast	& \ast & \ast
				%\end{pmatrix}
				&\geq\begin{pmatrix}
					0 & k_i &  k_i-1\\
					1	& 0 & 0 \\
					0	& 1 & 1
					\end{pmatrix}
								\begin{pmatrix}
									0 & k_{i+m_i} &  k_{i+m_i}-1\\
									1	& 0 & 0 \\
									0	& 1 & 1
								\end{pmatrix}\\
								&=\begin{pmatrix}
									\ast & \ast &  \ast\\
									0	& \ast & \ast \\
									\ast	& \ast & \ast
								\end{pmatrix}.
							\end{align*}
							We see that the multiplication from $k_i$ to $k_{i+m_i}$ does not result in a full matrix. By multiplying one more step, independently of the value of $k_{i+m_i+1}$ (it can be 1), we get
							\begin{align*}
								\begin{pmatrix}
									\ast & \ast &  \ast\\
									0	& \ast & \ast \\
									\ast	& \ast & \ast
								\end{pmatrix}A_{k_{i+m_i+1}}&= \begin{pmatrix}
									\ast & \ast &  \ast\\
									0	& \ast & \ast \\
									\ast	& \ast & \ast
								\end{pmatrix}\begin{pmatrix}
									0 & k_{i+m_i+1} &  k_{i+m_i+1}-1\\
									1	& 0 & 0 \\
									0	& 1 & 1
								\end{pmatrix}
								= \begin{pmatrix}	
									\ast & \ast &  \ast\\
									\ast	& \ast & \ast \\
									\ast	& \ast & \ast
								\end{pmatrix}
							\end{align*}
			and thus we have a full matrix by multiplying $m_i+2$ steps together. 
			
	For an arbitrary position $j$ with $k_j \geq 1$ there exists $n_j \geq 0$ such that there is $j\leq i \leq j+n_j$ with $k_{i} > 1$ and $m_i$ odd with $k_{i+m_i} > 1$. Thus the matrix associated to the substitution $ \chi_{k_j}\circ\cdots\circ \chi_{k_{i +m_i+1}}$ is strictly positive. Therefore the $S$-adic shift is primitive.
			
\textbf{Combinatorial Recognizability.}
Let $(X_i, \sigma)$ be the $S$-adic subshift based on $(\chi_{k_j})_{j\geq i}$, we show that substitution $\chi_{k_i}$ is combinatorially recognizable. Take $x\in X_i$, there is a unique way to decompose $x$ into blocks $\chi_{k_i}(a)$ for $a\in \mathcal{A}$. Any 2 in $x$ is its own block $\chi_{k_i}(1)$. Further every 3 in $x$ starts a new block and depending on the number of 1s directly following, we can determine if the block is $\chi_{k_i}(2)$ or $\chi_{k_i}(3)$.  For example
\begin{align*}
x 	&= \dots \ \vert\ 2 \ \vert \ 3\ 1\ 1\ \vert\ 3\ 1\ \vert\ 2\ \vert \ 2\ \vert  \ 3\ 1\ 1\ \vert\ \dots \\
	&= \dots \chi_2(1)\ \chi_2(2)\ \chi_2(3)\ \chi_2(1)\ \chi_2(1)\ \chi_2(2) \dots
\end{align*}
For any word $w$ with $\lvert w\rvert 1$ and $v$ such that $w$ appears twice in $\chi_{k_i}(v)$, if $w$ starts with letters 2 or 3 we know that there exists $i<j$ with 
$$
	\chi_{k_i}(v) = \chi_{k_i}(v_1\cdots v_i)w\cdots = \chi_{k_i}(v_1\cdots v_j)w \cdots.
$$
If $w$ starts with letter 1, then there exists no such $i,j$.

Thus the substitution $\chi_{k_i}$ is combinatorially recognizable and therefore the $S$-adic subshift based by substitutions $(\chi_{k_i})_{i\in\mathbb{N}}$ is combinatorially recognizable.

\textbf{Aperiodicity.} By Lemma~\ref{lm:aperiodic} there exist no periodic elements in $(X,\sigma)$.	 		
\end{proof}

From this proof, it is clear that a general method of telescoping
the sequence $A_{k_n}$ to obtain full matrices $\tilde A_i$ is the following:

\begin{equation}\label{eq:tele}
 \tilde A_i = \underbrace{A_1 \cdots A_1}_{r_{i,1}}
 											\cdot A_{k_{i,1}} \cdot \underbrace{A_1 \cdots A_1}_{r_{i,2}}
 \cdots \cdots
 A_{k_{i,m}} \cdot \underbrace{A_1 \cdots A_1}_{r_{i,m+1}}
 \cdot A_{k_{i,m+1}} \cdot A_{k_{i,m+2}},
\end{equation}
where $k_{i,j} \geq 2$ for $1\leq j \leq m+1$, $ k_{i,m+2} \geq 1$ and $r_{i,j}$ odd for $2 \leq j \leq m$, $r_{i,1}\geq 0$ and $r_{i,m+1}$ even.

\begin{remark}
	If the lengths of gaps between $k_i>1$ and $k_{i+m_i}>1$ with $m_i$ odd, is bounded for all $i$, then the subshift is strongly primitive, i.e., there exists a constant $M > \max_i(m_i)$ such that the matrix associated to the substitution $\chi_{k_{i}}\circ\cdots\circ \chi_{k_{i+2M}}$ is strictly positive for all $i$.
\end{remark}

\begin{lemma}\label{lm:notLR}
	Let $(X,\sigma)$ be a primitive non-periodic $S$-adic shift,
	based on substitutions $\chi_i:\mathcal{A}_i \to \mathcal{A}^+_{i-1}$.
	If for every $k \in \mathbb{N}$ there are $i < j$ and $a \in \mathcal{A}_{i-1}$, $b \in \mathcal{A}_j$ such that $a^k$ is a
	subword of $\chi_i \circ \cdots \circ \chi_j(b)$,
	then $(X,\sigma)$ is not linearly recurrent.
\end{lemma}
\begin{proof}
	Let $L \in \mathbb{N}$ be arbitrary and find $i < j$ and $b \in \mathcal{A}_j$
	such that $\chi_i \circ \cdots \circ \chi_j(b)$ has $a^{L+2}$ as subword.
	Set $w = \chi_1 \circ \cdots \circ \chi_{i-1}(a)$; then $X$ contains a sequence $x$ having $w^{L+2}$ as subword.
	Since $x$ is not periodic, the word-complexity $p_x(|w|) > |w|$, and hence there is
	a word $v$ of length $|v| = |w|$ that is not a subword of $w^{L+2}$.
	By primitivity, $v$ and $w^{L+2}$ appear infinitely often in $x$.
	But then $v$ must have reappearances with gap $\geq L |v|$.
	Hence $(X,\sigma)$ is not $L$-linearly recurrent.
	As $L$ was arbitrary, the lemma follows.
\end{proof} 

By Theorem 5.4 in \cite{BCY}, a $S$-adic subshift based on recognizable substitutions is linearly recurrent if and only if we can telescope the substitutions $(\chi_{k_i})_{i\geq 1}$ into finitely many, left-proper substitutions with strictly positive transition matrices.

\begin{thm}\label{thm:LR}
 The subshift $(X,\sigma)$ associated to an ITM $(\Omega,T_{\alpha,\beta})$ of infinite type is linearly recurrent
 if and only if $(k_i)_{i \in \N}$ is bounded and the sets
 $\{ i : k_{2i} > 1\}$ and $\{ i : k_{2i-1} > 1\}$  have bounded gaps.
\end{thm}

\begin{proof}
Let $M$ and $N$ be the maximal gap size of sets $\{i : k_{2i-1}>1\}$ and $\{i : k_{2i}>1 \}$ respectively. 
Every $M$-gap (i.e., $(k_n)_{n=i+1}^M$ for any $i\in\N$) contains at least one $k_{2i-1}>1$ and every $N$-gap at least one $k_{2i}>1$. Then there exists $\tilde{M} > \max\{M, N\}$ where for any $k_i\geq1$ there exists $k_{i + 2j-1} > 1$ and $k_{i + 2l} > 1$ for $ j,l < \tilde{M}$. %$\tilde{M}$ is the maximal distance between any $k_i >1$ and the next $k_j>1$, such that the parity of $j$ and $l$ is different, is bounded. We know that $\tilde{M}$ is the length of the maximal gap between $k_i > 1$ and $k_j>1$ with $i,j$ having a different parity.
Thus as in the proof of Proposition~\ref{prop:subshift} for any $i \in \mathbb{N}$ multiplying $2\tilde{M} + 1$ matrices together from $k_i$ to $k_{i+2\tilde{M}}$ will result in a full matrix. As there are only finitely many values $k_j \in \mathbb{N}$ can take, the set of full associated matrices is finite and the subshift  $(X, \sigma)$ is linearly recurrent.\\

On the other hand if $(k_n)_n$ unbounded we have that for any $L \in \mathbb{N}$ there exists $n\in\mathbb{N}$ with $k_n > L$ and therefore
$$
\chi_{k_n}(2) = 31^{k_n}
$$
contains the subword $1^{k_n}$.

Similarly if the gaps of $\{i : k_i > 1\}$ are unbounded, for any $L \in \mathbb{N}$ there exists a gap of length $r_n > 2L$ and
$$
\chi_{1}^{r_n}(2) = \left\{
\begin{array}{ll}
	3^{l_n}1 & \text{for $r_n = 2l_n-1$,} \\
	3^{l_n}2 & \text{for $r_n = 2l_n$} \\
\end{array}
\right. 
$$
contains the subword $3^{l_n}$.
By Lemma~\ref{lm:notLR} the subshift $(X,\sigma)$ is not linearly recurrent.
\end{proof}

\section{Weak mixing}\label{sec:WM}

In this section we  analyse the dynamics of the matrix products
$A_{k_1} \cdot A_{k_2} \cdots A_{k_i}$, $i \to \infty$, and ultimately related this to conditions under which the corresponding ITM is weak mixing or not.

\subsection{Lyapunov exponents}\label{sec:wm-Lyap}
Because the matrices $A_{k_i}$ have determinant $-1$, they can be viewed as automorphisms
$M_k:\T^3 \to \T^3$ of the $3$-torus $\T^3$, given by $M_k\vec x = \vec x A_k \bmod 1$.
Condition~\eqref{eq:wm2} can then be interpreted as that the line in $\T^3$ spanned by $(1,1,1)$ intersects the stable direction of $(0,0,0)$ for the infinite system
$( M_{k_i} \circ \cdots \circ M_{k_2} \circ M_{k_1})_{i \geq 1}$ in $\T^3$.
This explains the need of determining hyperbolicity and computing the dimension of the stable direction of this system in the first place.

\begin{prop}\label{prop:AB}
 For every sequence $(k_i)_{i \in \N}$ satisfying \eqref{eq:k2},
 the infinite matrix multiplication
$A_{k_1} \cdot  A_{k_2} \cdots$
has two positive and one negative Lyapunov exponent.
\end{prop}

In fact, we won't prove the existence of these Lyapunov exponents as limits,
but there are independent vectors $\vec v_1, \vec v_2, \vec v_3 \in \R^3$
and $0 < \lambda_3  <  1 < -\lambda_2 < \lambda_1$
%with $\lambda_1\lambda_2\lambda_3 = \det(A_k) = -1$
and $C > 0$ such that for all $n$,
\begin{equation}\label{eq:ev}
\| \vec v_i \bA^n\|
\geq C |\lambda_i|^n, \  i \in \{1,2\},
\quad \text{ and }\quad
\| \vec v_3 \bA^n\| \leq C^{-1} \lambda_3^n.
\end{equation}
where we write $\bA^n = \tilde A_1 \cdot \tilde A_2 \cdots \tilde A_n$ for the telescoped strictly positive matrices $\tilde A_i$  as outlined in \eqref{eq:tele}.
\\[3mm]
\begin{proof}
Let $\lambda_i(\bA^n)$, $i=1,2,3$, indicate the eigenvalues of $\bA^n$. Their product is $\pm 1$, because $\det(A_k) = -1$ for all $k$.
Each $A_k$ preserves the positive octant.
$Q^+ = \{ x \in \R^3 : x_i \geq 0\}$.  As in (the proof of) the Perron-Frobenius Theorem, % (see \cite{KH}),
 this leads to a positive leading Lyapunov exponent
$\liminf \frac1n \log \lambda_1(\bA^n) > 0$ and corresponding unstable space $E_1$ through the interior of $Q^+$.

The octant $Q^- = \{ x \in \R^3 : x_1,x_2 \geq 0 \geq x_3\}$
is preserved by each inverse matrix $A_k^{-1}$.
A change of coordinates changes $Q^-$ into $Q^+$. Indeed, take
\begin{equation}\label{eq:UB}
U := \begin{pmatrix}
  0 & 1 & 0 \\ 1 & 0 & 0 \\ 0 & 0 & -1
 \end{pmatrix}
 \quad \text{ and }\quad
 B_{k_i} = U A_{k_i}^{-1} U^{-1} =
 \begin{pmatrix}
  0  & 1 & k-1\\ 1 & 0 & 0 \\  0 & 1 & k
 \end{pmatrix}.
\end{equation}
Let $\bB^n = \tilde B_n \cdots \tilde B_2 \cdot \tilde B_1$, where the $\tilde B_i$'s are
blocks of matrices $B_{k_j}$ telescoped in exactly the same way (but in the other direction) as the $A_{k_i}$'s.
As will be shown in the proof of Lemma~\ref{lem:unique}  all the $\tilde B_i$'s
are strictly positive.

Then the Perron-Frobenius Theorem applies and we obtain
$\liminf_n \frac1n \log \lambda_3(\bB_n) > 0$.
This gives a negative $\limsup_n \frac1n \log \lambda_3(\bA_n) < 0$.

% Now for the remaining Lyapunov exponent, we need to show that $\liminf_n \frac1n \log |\lambda_2(A^n)| > 0$.
 Since $\det(\bA^n) = \pm 1$, we have
 $\log \lambda_1(\bA^n) + \log |\lambda_2(\bA^n)| + \log \lambda_3(\bA^n) = 0$ for every $n \in \N$.
 Therefore, if we can prove that  $\limsup_n \frac{1}{n} \left( \log \lambda_1(\bA^n)  + \log \lambda_3(\bA^n) \right) < 0$, then it follows that $\liminf_n \frac1n \log |\lambda_2(\bA^n)| > 0$.

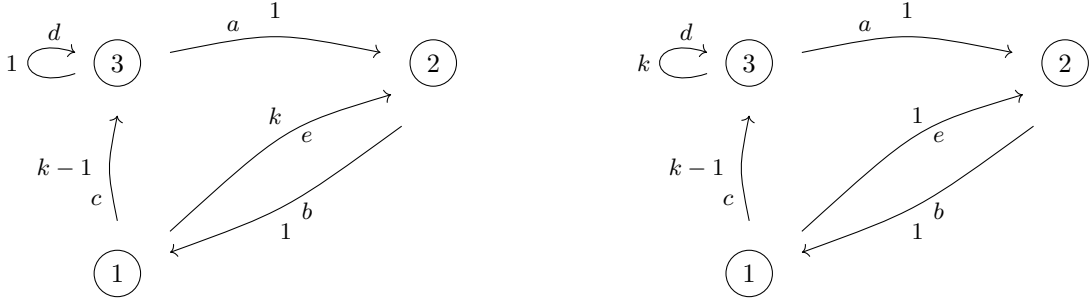
\begin{figure}[ht]
\begin{center}
\begin{tikzpicture}[scale=1.4]
\node[circle, draw=black] at (1,0.5) {1};
  \draw[->] (1,1) .. controls (0.9, 1.5) .. (1,2); \node at (0.5,1.5) {\small $k-1$};
   \node at (0.8,1.2) {\small $c$}; \node at (6.8,1.2) {\small $c$};
\node[circle, draw=black] at (4,2.5) {2};
  \draw[->] (3.7,1.9) .. controls (2.6, 1.1) .. (1.5,0.7); \node at (2.6,0.9) {\small $1$};
  \node at (2.8,1.1) {\small $b$}; \node at (8.8,1.1) {\small $b$};
   \draw[<-] (3.6,2.2) .. controls (2.6, 1.9) .. (1.5,0.9);  \node at (2.5,2) {\small $k$};
    \node at (2.8,1.8) {\small $e$}; \node at (8.8,1.8) {\small $e$};
\node[circle, draw=black] at (1,2.5) {3};
  \draw[->] (1.5,2.6) .. controls (2.5, 2.8) .. (3.5,2.6);  \node at (2.5,3) {\small $1$};
   \node at (2.1,2.85) {\small $a$}; \node at (8.1,2.85) {\small $a$};
  \draw[->] (0.6,2.4) .. controls (0,2.2) and (0,2.8) .. (0.6,2.6);  \node at (0,2.5) {\small $1$};
  \node at (0.4,2.8) {\small $d$};\node at (6.4,2.8) {\small $d$};
 %%%%%%
\node[circle, draw=black] at (7,0.5) {1};
  \draw[->] (7,1) .. controls (6.9, 1.5) .. (7,2); \node at (6.5,1.5) {\small $k-1$};
\node[circle, draw=black] at (10,2.5) {2};
  \draw[->] (9.7,1.9) .. controls (8.6, 1.1) .. (7.5,0.7); \node at (8.6,0.9) {\small $1$};
   \draw[<-] (9.6,2.2) .. controls (8.6, 1.9) .. (7.5,0.9);  \node at (8.6,2) {\small $1$};
\node[circle, draw=black] at (7,2.5) {3};
  \draw[->] (7.5,2.6) .. controls (8.5, 2.8) .. (9.5,2.6);  \node at (8.5,3) {\small $1$};
  \draw[->] (6.6,2.4) .. controls (6, 2.2) and (6, 2.8) .. (6.6,2.6);  \node at (6,2.5) {\small $k$};
\end{tikzpicture}
\caption{Transition graphs of $A_k$ and $B_k$. The numbers at the edge-labels stand
for weights (= multiplicities) of the edges.}
\label{fig:SFT}
\end{center}
\end{figure}

We will continue the proof using the original non-telescoped matrices $A_k$ and $B_k$.
For our choice of $B_k$, the corresponding transition graphs of
$A_k$ and $B_k$ are the same except that the multiplicities of the $12$-entry of $33$-entry are swapped.
(In particular, $A_k = B_k$ if $k=1$.)

The fact that $B_k$'s have a loop $3 \to 3$ with multiplicity $k$ (which has only multiplicity $1$ in $A_k$)
is the reason why $\bB^n$ is eventually (i.e., for sufficiently large $n$) larger than $\bA^n$ coordinate-wise.

Note that the matrix multiplication goes in opposite direction.
This can be remedied by taking the transpose, amounting to a reversal of arrows in the transition graph, which has no influence on the number of loops.

We claim that for every sequence $(k_i)_{i \in \N}$ there is $n_0$ such that for all $n \geq n_0$:
\begin{equation}\label{eq:AB}
 \sup_{u,v=1,2,3} (A_{k_1} \cdots A_{k_n})_{u,v} < \inf_{u,v=1,2,3} (B_{k_n} \cdots B_{k_1})_{u,v}.
\end{equation}
In fact, we can choose $n_0$ any integer larger than
$\min\{ j \geq 3 : J \text{ odd and } k_j \geq 2\}$, and hence we have
\begin{equation}\label{eq:ABtilde}
 \sup_{u,v=1,2,3} (\bA^n)_{u,v} < \inf_{u,v=1,2,3} (\bB^n)_{u,v}.
\end{equation}
The estimate \eqref{eq:AB} works for every sequence, so \eqref{eq:ABtilde} holds for every block.
Thus $\bA^n < \bB^n$ and also $\bA^n(1+\eps) < \bB^n$ entry-wise for some small
positive $\eps$.
It follows that
$(1+\eps)^n \lambda_1(\bA^n)  < \lambda_3(\bA^n)^{-1}$, proving that
$\limsup_n \frac{1}{n} \left( \log \lambda_1(\bA^n)  + \log \lambda_3(\bA^n) \right) \leq -\log(1+\eps) < 0$,
%and hence $\liminf_n \frac{1}{n}  \log \lambda_2(\bA^n)> 0$
as required.

It remains to prove claim \eqref{eq:AB}.
The $u,v$-entries of the matrices represent the number of $n$-paths from $u$ to $v$.
Note, that in every iterate (= matrix-multiplication) we use a different matrix,
of the same structure, but with different values $k_i$.
That is, we can use the same strings of labels, and insert (i.e., multiply with) the correct
multiplicities to compute the number of paths. Let us call this total number of paths for a single code
$x_1 \cdots x_n$ the weight $w(x_1 \dots x_n)$ of the code. The weight of a path is the product of the weights of its edges.

We will make a bijection between $n$-paths for $\bA^n$ and $\bB^n$ such that the multiplicity for
$\bB^n$ is always higher.

We first look at the loops from edge $a$ back to $a$. There are represented by words $a \cdots a$ (allowed by the transition graph),
with no $a$ in between. In fact, we don't write the first $a$ because it is the end of the previous loop.

Let $w_A$ and $w_B$ denote the weights of the edges of the transition graphs of $A$ and $B$.
 For every $n \geq 1$, we have
\begin{equation}\label{eq:wAB}
 \sum_{j=0}^{n-1} w_A( (be)^jbc d^{2(n-j-1)} )
 = k_2k_4 \cdots k_{2n}-1 \leq
  \sum_{j=0}^{n-1} w_B( (be)^jbc d^{2(n-j-1)} ),
\end{equation}
and the inequality is strict
if any of the odd-indexed $k_{2j+1} \geq 2$.

Each edge $b$ has weight $w_A(b) = w_B(b) = 1$;
the edges $c$ in position $2i$ have weight $w_A(c) = w_B(c) = k_{2i}-1$
and the remaining edges in position $i$ have weight
$w_A(e) = w_B(d) = k_i$, $w_A(d) = w_B(e) = 1$.

We prove by induction that
$\sum_{j=0}^{n-1} w_A( (be)^jbc d^{2(n-j-1)} ) =
k_2k_4 \cdots k_{2n} - 1$.
Indeed, for $n = 1$ we have $w_A(bc) = k_2-1$.
If the statement holds for $n-1$, then
\begin{eqnarray*}
 \sum_{j=0}^{n-1} w_A( (be)^jbc d^{2(n-j-1)} ) &=&
 \sum_{j=0}^{n-2} w_A( (be)^jbc d^{2(n-j-1)} ) +
 k_2k_4 \dots (k_{2n}-1) \\
 &=&
 k_2k_4 \cdots k_{2n-2} - 1 +  k_2k_4 \dots (k_{2n}-1) \\
 &=& k_2k_4 \cdots k_{2n}-1.
\end{eqnarray*}
We show that
$\sum_{j=m}^{n-1} w_B( (be)^jbc d^{2(n-j-1)} ) =
k_{2m+2} \cdots k_{2n} - 1$  for all $n \in \N$ and $0 \leq m < n$.
This is done by induction too,
but now working downwards, first taking the odd-indexed
$k_{2j+1} = 1$.
For $m = n-1$, we get $w_B((be)^{n-1}bc ) = k_{2n}-1$.
If the statement holds for $m$, then
\begin{eqnarray*}
 \sum_{j=m-1}^{n-1} w_B( (be)^jbc d^{2(n-j-1)} ) &=&
 \sum_{j=m}^{n-1} w_B( (be)^jbc d^{2(n-j-1)} ) +
 k_{2m}k_{2m+2} \dots (k_{2n}-1) \\
 &=&
 k_{2m+2} \cdots k_{2n-2} - 1 +
 k_{2m} \dots (k_{2n}-1) \\
 &=& k_{2m} \cdots k_{2n}-1.
\end{eqnarray*}
This concludes the induction.
Since each instance $w_B(d)>1$ at an odd position increases
the terms in the sum of $w_B$-weights
by a factor $\geq 2$, the inequality follows.
This proves \eqref{eq:wAB}.

In continuation of the proof of \eqref{eq:AB},
the bijection between loops is done by grouping according to their lengths and
 estimating their weights group-wise using equation~\eqref{eq:wAB}.

For paths that don't start or finish in $a$, similar arguments hold,
where we can offset $(be)^j$ in the graph of $A$ by $(dd)^j$ in the graph of $B$, so that indeed
$\sup_{u,v=1,2,3} (\bA^n)_{u,v} < \inf_{u,v=1,2,3} (\bB^n)_{u,v}$
 for sufficiently large $n$.
\end{proof}

\subsection{Host's eigenvalue condition}\label{sec:host}
Host \cite{Host86} formulated a condition for primitive substitution shifts
to have an eigenvalue:
For $e^{2\pi i \ev}$
to be an eigenvalue the Koopman operator for some $\ev \in (0,1)$ is:
\begin{equation}\label{eq:wm}
\sum_{n=1}^\infty \tb  \vec \ev A^n \tb  < \infty,
\qquad \vec \ev = (\ev,\ev,\ev),
\end{equation}
where $\tb x \tb$ is the distance of a vector to the nearest integer lattice point and $A$ is the associated matrix of the substitution.
The condition is of the same gist as Veech criterion  \cite{V84}
for non-weak-mixing of
IETs, although there $A$ can be seen as the matrix representing the action of Rauzy induction on the homology $H^1(\Z)$ of the translation surface obtain as suspension flow over the IET.

For primitive substitution shifts, \eqref{eq:wm} is both sufficient and necessary and every eigenfunction can be taken to be continuous.
Later works, see e.g.\  \cite{CDHM03,DFM19} showed that
in the context of linearly recurrent $S$-adic shifts (i.e., subshifts based on a sequence of substitution rather than a single one), \eqref{eq:wm} is necessary and sufficient
as well. For more general $S$-adic subshifts, more complicated conditions are needed, and we come back to them in Sections~\ref{sec:wm-continuous} and~\ref{sec:wm-measurable} where we distinguish between continuous and measurable eigenvalues.

\subsection{Weak mixing for (pre-)periodic ITMs}\label{sec:wm-periodic}

\begin{thm}\label{thm:periodic}
 For every (pre-)periodic sequence $(k_i)_{i \in \N}$ satisfying \eqref{eq:k2},
 the corresponding system $(\Omega,T_{\alpha,\beta})$ is weakly mixing.
\end{thm}

\begin{proof}
Suppose that the sequence $(k_i)_{i \in \N}$ has period $n$ and also satisfies \eqref{eq:k2},
then the corresponding $S$-adic shift is linearly recurrent (and uniquely ergodic).
If $A^n = A_{k_1} \cdots A_{k_n}$,
then its eigenvalues are $\tilde \lambda_i$ where $\tilde \lambda_3 < 1 < |\tilde \lambda_2| < \tilde \lambda_1$.
That is, there is only a one-dimensional stable subspace $E_3$ spanned by a vector $(u_1,u_2,-1) \in Q^-$.
Also, since $A^n$ is irreducible, the $\tilde \lambda_i$ are cubic numbers.

As $(k_i)$ is periodic, telescoping gives a stationary sequence,
so condition \eqref{eq:wm} decides on the eigenvalues of the
the Koopman operator.
To prove that there is a non-trivial eigenvalue, we need to show that $\vec \ev$ lies in some integer translation of $E_3$. That is
\begin{equation}\label{eq:tt}
(\ev,\ev,\ev)+s(u_1,u_2,-1) = (p,q,r)
\end{equation}
for some integers $p,q,r$ and reals $u_1,u_2 >0$.
Eliminating $s$ and $\ev$ from this system, we obtain
%that $u_1+1$ and $u_2+1$ are rationally dependent:
 \begin{equation}\label{eq:u1u2}
  (q-r)u_1 = (p-r)u_2+(p-q).
 \end{equation}
 For the matrix $A^n = (a_{ij})_{i,j=1}^3$, the third eigenvalue equation
 $\tilde\lambda_3 (u_1,u_2,-1) = (u_1,u_2,-1) A$
 gives
 \begin{eqnarray*}
  \tilde\lambda_3 u_1 &=& a_{11} u_1 + a_{21} u_2 - a_{31},\\
   \tilde\lambda_3 u_2 &=& a_{12} u_1 + a_{22} u_2 - a_{32}.
 \end{eqnarray*}
 Multiplying these equations and inserting \eqref{eq:u1u2} to eliminate
 $u_1$, we obtain
 \begin{eqnarray*}
  \tilde\lambda_3( (p-r)u_2 + (p-q)) &=& a_{11}((p-r)u_2+(p-q)) + a_{21}(q-r)u_2 - (q-r)a_{31},\\
   \tilde\lambda_3 (q-r)u_2 &=& a_{12}((p-r)u_2+(p-q)) + a_{22}(q-r)u_2 - (q-r)a_{32}.
 \end{eqnarray*}
Making $u_2$ subject of the second equation (so that the RHS is a fractilinear expression in $\tilde\lambda_3$) and then inserting it in the first equation gives an equation of two fractilinear expressions in $\tilde\lambda_3$.
Therefore $\tilde\lambda_3$ is the solution of a quadratic equation
contradicting that $\tilde\lambda_3$ is a cubic number.

For the preperiodic case, with preperiod $m$,
we need to replace \eqref{eq:tt} by
$$
(\ev,\ev,\ev) \cdot A_{k_1} \cdots A_{k_m} + s(u_1,u_2,-1) = (p,q,r).
$$
This gives a more cumbersome version of \eqref{eq:u1u2}
but the argument is essentially the same.
\end{proof}

\subsection{The stable direction.}\label{sec:wm-stable}

Let $(\Omega, T_{\alpha,\beta})$ be an infinite type ITM based on a sequence $(k_n)_{n \in \N}$ satisfying \eqref{eq:k2}.
We call $W^s(\vec 0) := \bigcap_n A_{k_1}^{-1} \circ \dots \circ A_{k_n}^{-1}(Q^{-}) = \Span(\vec v_3)$ the {\em stable space} of the sequence $(A_{k_i})_{i \in \N}$.
In order to find  $W^s(\vec 0)$, we iterate the matrices $B_{k_i}$ from~\eqref{eq:UB}. This gives
$$
(u,v,w) = \lim_{n \to \infty} \frac{B^n(\vec a)}{\| B^n(\vec a)\|_1}
\quad \text{ for } \quad
B^n = B_{k_n} \circ \cdots \circ B_{k_1}
$$
and, apart from the assumption of unique ergodicity, the choice of the vector $\vec a \in Q^+$ can be arbitrary.
Then $(v,u,-w) = (u,v,w) U$ (with $U$ from \eqref{eq:UB})
is the stable direction of $A_{k_1} \cdot A_{k_2} \cdots$,
normalised so that $u+v+w = 1$.

\begin{lemma}\label{lem:unique}
 The direction $(u,v,w)$ is uniquely determined by the sequence $(k_i)_{i \in \N}$,
 provided it satisfies \eqref{eq:k2}.
\end{lemma}
Thus, even if $T_{\alpha,\beta}$ fails to be uniquely ergodic and there is no unique
unstable direction in the first octant, the stable direction in $Q^-$ is always well-defined.
\\[2mm]
\begin{proof}
 Since $(k_i)_{i \in \N}$ satisfies \eqref{eq:k2},
 there are infinitely many $i$ and integers $r \geq 0$ such that
 $k_i \geq 2$, $k_{i-1} = \dots = k_{i-2r} = 1$ and $k_{i-2r-1} \geq 2$.
 Abbreviate $a = k_i \geq 2$, $b = k_{i-2r-1} \geq 2$ and $c = k_{i-2r-2} \geq 1$. The telescoped block
 from $k_i$ to $k_{i+2r+2}$ gives
 \begin{eqnarray*}
 \tilde B &=& B_a \cdot B_1^{2r} \cdot B_b \cdot B_c\\
 &=&
 \begin{pmatrix}
  (a-1)(r+1) & (a-1)b(r+1)+1 & c( (a-1)b(r+1)+1 ) - (r(a-1)+1) \\
  1 & b-1 & c(b-1) \\
  a(r+1) & ab(r+1)+1 & c(ab(r+1)+1) -(ra+1)
 \end{pmatrix},
\end{eqnarray*}
which is a strictly positive matrix. Therefore it represents a strict contraction in the Hilbert metric on the first octant, see e.g.\ \cite[Section 8.6]{B23}.
The contraction factor is bounded by
$\tanh(\frac12 \log(\rho))$, where\footnote{We phrase this different from formula (8.29) in \cite{B23} because we are using left-multiplication with row vectors instead of right multiplication with column vectors as in \cite{B23}.}
\begin{equation}\label{eq:rho}
\rho = \max_{1 \leq j,j' \leq 3} \sqrt{\frac{\max\{\tilde B_{j,k}/\tilde B_{j',k} : 1 \leq k \leq 3\} }
{\min\{\tilde B_{j,k}/\tilde B_{j',k} : 1 \leq k \leq 3\}} }.
\end{equation}
The shape of $\tilde B$, where the rows are almost multiples of each other,
yields that $\rho$ is bounded, and hence
the contraction factor smaller than $1$, uniformly in $a,b,c$.
As there are infinitely many such telescoped blocks $\tilde B$,
the infinite matrix product contracts the positive octant into a single half-line $\ell$.
Thus $U\ell$ for the matrix $U$ from \eqref{eq:UB} represents the unique stable direction $W^s(\vec 0)$.
\end{proof}

To find an eigenvalue, we have to solve
\begin{equation}\label{eq:ttt}
(\ev,\ev,\ev) = (p,q,r) + s(v,u,u+v-1) \text{ for some } p,q,r \in \Z.
\end{equation}
Solving for $\ev$ and $s$ gives arcs
$\ell_{p,q,r} = \{ (u,v) \in \Delta : (1-u)(r-q) = (1-v)(r-p)\}$ in the simplex
$\Delta = \{ (u,v) : 0 \leq u \leq 1-v\}$, or equivalently
\begin{equation}\label{eq:ttt2}
\ell_{p,q,r} = \{ (u,v) \in \Delta : u(q-r) = v(p-r)+q-p\}.
\end{equation}
Expressed in terms of $p,q,r,\ev$, we find
\begin{equation}\label{eq:uv}
 u = \frac{\ev-q}{\ev+r-p-q}, \qquad v = \frac{\ev-p}{\ev+r-p-q},
\end{equation}
so that $\ev = \frac{r-p}{1-u} + p+q-r = \frac{r-q}{1-v} + p+q-r \in \Q$ if and only if both $u, v \in \Q$. We define $\ell_{p,q,r}(\ev)$ as those points $(u,v) \in \Delta$ such that \eqref{eq:uv} holds.
Let $H_k(u,v)$ indicate the first two coordinates of
$(u,v,1-(u+v)) B_k$ normalised to unit length.
Then $H_k:\Delta \to \Delta_k := H_k(\Delta)$ is a map from the unit simplex $\Delta$, and it has the formula
\begin{equation}\label{eq:H}
H_k:(u,v) = \frac{1}{D_k} (v\ ,\ 1-v)
\quad \text{ for } \quad D_k = k(1-v)+1-u,
\end{equation}
\begin{equation}\label{eq:Hinv}
H_k^{-1}: (x,y) \mapsto \left(\frac{x+(k+1)y-1}{x+y}\ ,\ \frac{x}{x+y}\right) \quad \text{ for } \quad (x,y) \in \Delta_k.
\end{equation}
The derivative and its determinant are
$$
DH_k(u,v) = \frac{1}{D_k^2}\begin{pmatrix}
                    v & k+1-u \\ 1-v &  u-1
                   \end{pmatrix},
\qquad
\det DH_k(u,v) = -\frac{1}{D_k}.
%\quad \text{ for } \ D_k = k(1-v) + 1-u.
$$

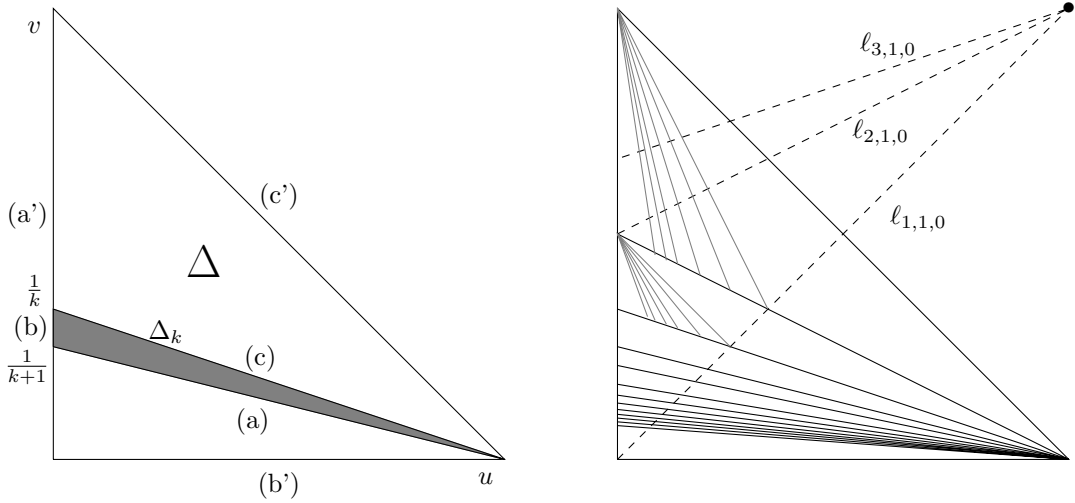
\begin{figure}[ht]
\begin{center}
\begin{tikzpicture}[scale=0.5]
\draw[-] (0,0)--(12,0)--(0,12)--(0,0);
\node at (11.5, -0.5) {$u$}; \node at (-0.5, 11.5) {$v$};
\filldraw[gray] (0,3)--(12,0)--(0,4)--(0,3);
\draw[-] (0,3)--(12,0)--(0,4)--(0,3);
\node at (4,5.26) {\LARGE $\Delta$}; \node at (3,3.4) {$\Delta_k$};
\node at (-0.5,4.5) {$\frac1{k}$};\node at (-0.7,2.45) {$\frac1{k+1}$};
\node at (-0.6,3.5) {(b)};\node at (5.3,1) {(a)};\node at (5.5,2.7) {(c)};
\node at (-0.685,6.5) {(a')};\node at (6,-0.7) {(b')};\node at (6,7) {(c')};
\draw[dashed] (15,6)--(27,12)--(15,8);
\draw[dashed] (15,0)--(27,12);
\node at (27,12) {$\bullet$};
\node at (22,8.7) {$\ell_{2,1,0}$};
\node at (22.2,11) {$\ell_{3,1,0}$};
\node at (23,6.4) {$\ell_{1,1,0}$};
%%%%%
\draw[-] (15,0)--(27,0)--(15,12)--(15,0);
\draw[-] (15,6)--(27,0)--(15,4);
\draw[-] (15,3)--(27,0)--(15,2.5);
\draw[-] (15,2)--(27,0)--(15,1.71);
\draw[-] (15,1.5)--(27,0)--(15,1.33);
\draw[-] (15,1.2)--(27,0)--(15,1.1);
\draw[-] (15,1)--(27,0)--(15,0.9);
\draw[-, color=gray] (19, 4)--(15,12)--(18, 4.5);
\draw[-, color=gray] (17.2, 4.9)--(15,12)--(16.6, 5.2);
\draw[-, color=gray] (16.3, 5.3)--(15,12)--(16, 5.44);
\draw[-, color=gray] (15.8, 3.8)--(15,6)--(18, 3);
\draw[-, color=gray] (17.2, 3.3)--(15,6)--(16.6, 3.5);
\draw[-, color=gray] (16.3, 3.6)--(15,6)--(16, 3.7);
\end{tikzpicture}
\caption{The simplex $\Delta$ and image $\Delta_k = H_k(\Delta)$ (left) and further images (right).}
\label{fig:Deltak}
\end{center}
\end{figure}

To find the set $\Delta_k$, we compute the images of the three boundary lines of $\Delta$:
\begin{eqnarray*}
 (a') \to (a) && H_k( \{ 0 \} \times [0,1]) =
\Big\{ \frac{1}{k(1-v)+1} (v,1-v) : v \in [0,1]\Big\},\\
 (b') \to (b) && H_k([0,1] \times \{ 0 \}) =
\Big\{ \frac{1}{k+1-u} (0,1) : u \in [0,1]\Big\},\\
 (c') \to (c) && H_k(\{ (1-v,v) : v \in [0,1]\}) =
\Big\{ \frac{1}{k(1-v)+v} (v,1-v) : v \in [0,1]\Big\},\\
\end{eqnarray*}
see Figure~\ref{fig:Deltak}, left. Comparing (a) and (c) we can see that the triangles $\Delta_k$ and $\Delta_{k+1}$
are adjacent, and $\Delta_1$ is adjacent to the upper boundary of $\Delta$. That is, the $\Delta_k$'s have disjoint interiors and
$\Delta  = \overline{ \bigcup_{k\geq 1} \Delta_k}$.
Furthermore, unless $k=k'=1$, $H_k \circ H_{k'}$ maps $\partial \Delta$ into the interior of $\Delta$, and hence, for every $(u,v) \in \Delta$,
there are no common boundaries of $\Delta_k$'s that are the image of
any point $H_k \circ H_{k'} \circ H_{k''}$....
It follows that for each $(u,v) \in \Delta$,  there is at most one sequence $(k_i)_{i \in \N}$
such that $\lim_{n\to\infty} H_{k_1} \circ H_{k_2} \circ \cdots \circ H_{k_n}(\vec a) = (u,v)$.

The boundary points of the $\Delta_k$ cannot be reached as limit  from an interior point $(u,v) \in \Delta^\circ$,
and under later iterates of the $H_k$'s, the iterated boundary points
$E := \bigcup_{n\in \N} \bigcup_{k_1, \dots, k_n} H_{k_1} \circ \cdots \circ H_{k_n}(\partial \Delta)$
cannot be equal to any point in $\bigcap_n  \bigcup_{k_1, \dots, k_n} H_{k_1} \circ \cdots \circ H_{k_n}(\Delta^\circ)$.

\begin{remark}
	The points $(x,y) \in \Delta$ which give rise to parameters $(k_n)_{n\in\N}$ such that the subshift associated is linearly recurrent, are bounded away
	from $\partial\Delta$ for all preimages $H_{k_n}^{-1} \circ \cdots \circ H_{k_1}^{-1}(x,y)$. For $(k_n)_{n\in\N}$ unbounded, a subsequence of preimages goes to the edges $(a')$, $(b')$; in case of unbounded blocks of $1$'s in $(k_n)_{n\in\N}$, a subsequence of preimages goes to $(c')$.
	Conjecturally, this set has full Hausdorff dimension, but for each fixed bound
	$C$,	the set of parameters for which $\limsup_n k_n \leq C$ and
	the lengths of blocks of $1$'s is eventually bounded by $C$
	has Hausdorff dimension strictly less than $1$. Since the maps $H_k$ are non-conformal, the current techniques of fractal geometry seem insufficient to prove this.
\end{remark}

The maps $H_k$ send points with rational coordinates to  points with rational coordinates,
and hence the lines between such points (hence with rational slope) to lines of the same properties.

\begin{prop}\label{prop:slope}
 The sides of every triangle $H_{k_1} \circ \cdots \circ H_{k_n}(\partial \Delta)$ have negative slopes for all $n \geq 1$
 and $k_1, \dots, k_n \in \N$.
\end{prop}

\begin{proof}
 With some stretch of terminology, the statement holds for $n=0$, so for $\Delta$ itself:
 here the ``bottom'' side $\partial_b\Delta$ has slope $0$ and the left side $\partial_l\Delta$ needs to be interpreted as with slope $-\infty$.
 Only the ``right'' side $\partial_r\Delta$ has a proper negative slope $-1$.

 Now $\bigcup_{k \in \N} \partial \Delta_k$ consists
 of $\partial_b\Delta$ and a fan $F$ of
 straight lines stretching from $(1,0)$ to $\partial_l\Delta$,
 see Figure~\ref{fig:Deltak}, right.
 The slope of these lines are therefore all negative.

 The maps $H_k$ reverse orientation: they are reflections in lines of positive slope combined with a (non-linear) contraction.
 Since the vertex $V_k := (0,\frac{1}{k}) = H_k(1,0)$ of $\Delta_k$ is the intersection of
 $\partial_l\Delta$ and $\partial_r\Delta_k$, $H_k(F)$ is a fan
 of lines stretching from $V_k$ to the opposite boundary $\partial_b\Delta_k$. These lines have slopes between the slopes of
 $\partial_l\Delta_k$ and $\partial_r\Delta_k$, so they are negative.

The proof follows from repeating this argument inductively.
\end{proof}

\subsection{Continuous eigenvalues for the general case.}\label{sec:wm-continuous}

Let $(E_n,V_n,\succ)$ be an ordered Bratteli diagram with Vershik transformation $\tau$ and associated matrices $A_i$.
(See \cite{DFM19} or \cite[Chapter 5.4]{B23} for the precise definitions.)
Theorem 2 in \cite{DFM19} states for {\bf positive} transition matrices $M_j$, $j \geq 1$, that $e^{2\pi i \ev}$ is a continuous eigenvalue if and only if
\begin{equation}\label{eq:wm2}
\sum_{n=1}^\infty \max_{x \in X} \tb \langle s_n(x) , \vec \ev M_1 \cdots M_n \rangle\tb  < \infty,
\qquad \vec \ev = (\ev,\ev,\ev),
\end{equation}
where  $(s_n(x))_v = \# \{ e\in E_{n+1} \ : \ e \succ x_{n+1}, \fs(e) = v \}$. Thus in the ordered Bratteli-Vershik diagram (with $\succ$ indicating the
order of the incoming edges)
the vector $s_n(x)$ counts the number of incoming edges that are higher in the order than edge $x_{n+1}$ in the path $x$.
This definition of $s_n$ is phrased in terms of ordered Bratteli diagrams. In terms of $S$-adic shifts $(X,\sigma)$ based on substitutions $\chi_n:\cA_n \to \cA_{n-1}$,
every $x \in X$ can be written as
$$
x = \lim_{n \to\infty} \sigma^{j_1} \circ \chi_1 \circ \sigma^{j_2} \circ  \chi_2 \dots
\sigma^{j_n} \circ  \chi_n(a_n)
$$
for some integer sequence $(j_n)$ and symbols $a_n \in \cA_n$
such that $a_n$ is the first letter of $\sigma^{j_{n+1}} \circ \chi_{k_{n+1}}(a_{n+1})$.
Then
$$
s_{n,b}(x) = |\sigma^{j_{n+1}} \circ \chi_{k_{n+1}}(a_{n+1})|_b,
\qquad b \in \cA_n
$$
where $|w|_b$ indicates the number of symbols $b$ in a word $w$.

We will use this for our situation, so with $\chi_{k_i}$ and $A_{k_i}$
(or in fact, the telescoped version $\tilde A_i$ from \eqref{eq:tele}) replacing $\chi_i$ and $A_i$.
The paper \cite{DFM19} is formulated for invertible $S$-adic shifts, but
the Bratteli-Vershik system having a minimal path (which is true in our case because $\chi_{k_i} \circ \chi_{k_{i+1}}$ is left-proper)
is sufficient.

\begin{lemma}
 Let $\{ \tilde A_i\}_{i \geq 1}$ be the telescoped version of $(A_i)_{i \geq 1}$ given in \eqref{eq:tele},
 and $\tilde s_{i+1}$ the associated $s$-value of $\tilde A_{i+1}$.
 Then
 \begin{equation}\label{eq:sum}
  \sum_{n=1}^\infty \max_{x \in X} \tb \langle s_n(x), \vec \ev A_{k_1} \cdots A_{k_n} \rangle\tb
  \geq
  \sum_{n=1}^\infty \max_{x \in X} \tb \langle \tilde s_n(x),  \vec \ev \tilde A_1 \cdots \tilde A_n \rangle\tb.
 \end{equation}
\end{lemma}

\begin{proof}
Let $ \tilde h_n = (1,1,1)\tilde A_1 \cdots \tilde A_n$.  We can decompose the telescoped $ \langle \tilde s_i(x),  \tilde h_i \rangle$ back into levels $n_i$ (the index of the first matrix in the product $\tilde A_i$) to $n_{i+1}-1$ (the index of $A_{k_{i,m+2}}$). It follows that
$$
	\langle \tilde s_i(x),  \tilde h_i \rangle = \sum_{j = n_i}^{n_{i+1}-1} \langle s_j(x), h_j \rangle.
$$
Then
\begin{align*}
	%\sum_{n=1}^\infty \max_{x \in X} \tb \tilde s_n(x) \vec \ev \tilde A_{k_1} \cdots \tilde A_{k_n} \tb
	\sum_{i=1}^\infty \max_{x \in X} \tb \ev \langle \tilde s_i(x), \tilde h_i\rangle\tb
	&=  \sum_{i=1}^\infty \max_{x \in X} \tb \ev \sum_{j = n_i}^{n_{i+1}-1} \langle s_j(x), h_j\rangle\tb\\
	&\leq \sum_{i=1}^\infty \max_{x \in X}  \sum_{j = n_i}^{n_{i+1}-1} \tb \ev\langle s_j(x), h_j\rangle\tb\\
	&\leq \sum_{j=1}^\infty \max_{x \in X} \tb \ev\langle s_j(x), h_j\rangle\tb,
\end{align*}
as claimed.
\end{proof}
\begin{remark}
For our specific case of infinite type ITMs $s_j(x)$ can be simplified to the vector $(r_j,0,0)^T$ with $r_j \in\{0,\dots k_j \}$ depending on the choice of $x$. Then the sum is
$$
\sum_{j=1}^\infty \max_{x \in X} \tb \ev\langle s_j(x), h_j\rangle\tb =  \sum_{j=1}^\infty \max_{r_j \in\{0,\dots k_j\}} \tb \ev r_j h_j(1)\tb.
$$
\end{remark}

One would expect (in accordance with the Veech criterion \cite{V84})
the Koopman operator to have a non-trivial eigenvalue if and only if  $(u,v) \in \ell_{p,q,r}$ for some $p,q,r\in \Z$ and otherwise the map is weak mixing.
However, also because we have to deal with summability
condition \eqref{eq:wm2}, especially the factor $s_n(x)$, the story is not as clear-cut as that.
In one direction, it is as expected:

\begin{thm}\label{thm:wm}
 If $\vec \ev$ does not belong to the stable space
 $W^s(\vec 0) \bmod 1$ for $\ev \in (0,1)$, then the corresponding ITM has no continuous eigenvalue.
\end{thm}

\begin{proof}
Let $\vec \ev_0 = \vec \ev \bmod 1$ and $\vec \ev_{n+1} = \vec \ev_n \tilde A_{n+1} \bmod 1 \in (-\frac{1}{2},\frac{1}{2}]^3$. Assume by contradiction that
condition \eqref{eq:wm2} holds, i.e., $\max_{x\in X} \tb \langle \tilde s_{n}(x) , \vec \ev_n \rangle \tb$ is summable.
Take $n_0 \in \N$ so that $\max_{x\in X} \tb \langle \tilde s_{n}(x) , \vec \ev_n \rangle \tb <  1/4$ for all $n \geq n_0$.

The vectors $\tilde s_{n}(x)$ have non-negative entries, but are smaller
than the $v$-th column of $\tilde A_{n+1}$ if the path $x \in X$ is such that $r(x_{n+1})=v$.
Using the language of BV-diagrams, if $x_{n+1}$ is the smallest incoming edge to $v$, then $\tilde s_{n}(x)$ has the largest possible value,
say $s^+_{n}(x)$, which is the $v$-th column of $\tilde A_{n+1}$ with one entry decreased by $1$.
Taking successive incoming edges for $x_{n+1}$, $\tilde s_{n}(x)$ decreases each time by
one unit, depending on the source vertex $\fs(x_{n+1})$.

Now if $\vec \ev_n \in B_\eps(\vec 0)$, but $\langle \tilde s^+_{n}(x), \vec \ev_n\rangle$
is close to a non-zero integer, then choosing paths $x$ with $x_{n+1}$ equal to successive
incoming edges to $v$,  $\langle \tilde s_{n}(x), \vec \ev_n\rangle$ decreases every step by an
amount $<\eps$.
At some step,  $\tb \langle \tilde s_{n}(x), \vec \ev_n\rangle \tb > \frac13$,
contradicting that $\max_{x\in X} \tb \langle \tilde s_{n}(x) , \vec \ev_n \rangle \tb < 1/4$.

This implies that $\vec \ev_n \cdot \tilde A_{n+1} \cdots \tilde A_{n+m} \to \vec 0$ as $m \to \infty$, contrary to the assumption that $\vec \ev \notin W^s \bmod 1$, and the proof is complete.
\end{proof}

The next proposition together with \eqref{eq:uv} shows that there cannot be an eigenvalue $e^{2\pi i \ev}$
for rational $\ev$, because $\frac{1}{\ev+r-p-q} (\ev-q,\ev-p) \in \Delta$ is contained in
 $H_{k_1} \circ \cdots \circ H_{k_n}(\partial \Delta)$
for $n$ sufficiently large, and that means there is no sequence $(k_i)_{i \in \N}$ satisfying \eqref{eq:k2}
such that  $\ell_{p,q,r}(\ev) = \lim_{n \to \infty} H_{k_1} \circ \cdots \circ H_{k_n}(\Delta^\circ)$.

\begin{prop}\label{prop:rat}
 Every rational point in $\Delta$ belongs to
 $\bigcup_{n \geq 0} \bigcup_{k_1, \dots, k_n \in \N}
 H_{k_1} \circ \cdots \circ H_{k_n}(\partial \Delta)$.
\end{prop}

\begin{proof}
 Suppose $(x,y) \in \Delta^\circ$ has rational coordinates.
 We can give them the same denominator,
 i.e., we write $(x,y) = (\frac{p}{q}, \frac{p'}{q})$
 with $p+p' < q$ because $x+y < 1$.
 Take $k_1 \in \N$ such that $(x,y) \in H_{k_1}(\Delta)$.
 If $(x,y) \in H_{k_1}(\partial \Delta)$ we are done, so assume that
 $(x,y) \in H_{k_1}(\Delta^\circ)$.
 Then using \eqref{eq:Hinv} we get
$$
 H_{k_1}^{-1}(x,y) = \left(\frac{x+(k_1+1)y-1}{x+y}\ ,\ \frac{x}{x+y}\right) = \left( \frac{ p + (k_1+1)p' - q }{p+p'} \ , \ \frac{p}{p+p'} \right),
 $$
 which are fractions of common denominator $p+p' < q$,
 independently of $k_1$.
 Continuing this way, we find $n < q$ such that
 $H_{k_n}^{-1} \circ \cdots \circ H_{k_1}^{-1}(x,y)$ lies on the line $\{ x+y=1 \} \subset \partial \Delta$.
\end{proof}

\begin{thm}
 There exist parameters $(\alpha,\beta)$ such that $T_{\alpha,\beta}$ is of type $\infty$ and $\vec \ev $ does belong to the stable space
 $W^s(\vec 0) \bmod 1$ for $\ev \in (0,1)$, but $e^{2\pi i \ev}$ is not a continuous eigenvalue of the Koopman operator.
\end{thm}

\begin{proof}
From \eqref{eq:ttt2} it follows that a parameter $(u,v) \in \Delta$ is such that $\vec \ev$ is in the stable subspace $W^s(\vec 0)$ only if it belongs to  $\ell_{p,q,r}$ for some $p,q,r \in \Z$.
 The lines $\ell_{p,q,r}$ all pass through $(1,1)$ and have positive
 slope if they indeed intersect $\Delta$.
 Therefore if $\ell_{p,q,r}$ intersections some subtriangle
 $H_{k_1} \circ \cdots \circ H_{k_n}(\Delta)$, it goes through its upper side.

 Assume now that $\tilde A_1 \cdots \tilde A_i$
 is a block of telescoped matrices, each of the form \eqref{eq:tele}.
 Then $\eps_i := \tb \vec \ev A_{k_1} \cdots A_{k_n}\tb$ is exponentially small in $i$ (or even smaller).
 For the block $\tilde A_{i+1}$ as in \eqref{eq:tele},
 we choose $m=1$, and let $r_{i+1,1}$ be some large integer and
 $k_{i+1,1} =k_{i+1,2} = 2$, then
 $\tilde A_{i+1} = A_1^r \cdot A_{k_{i+1,1}} \cdot A_{k_{i+1,2}} \cdot A_{k_{i+1,3}}$ is the next telescoped matrix, and the corresponding value $\tilde s_{i} \approx r$.
 Choosing $r$ sufficiently large, we can assure that
 $\frac{1}{i} \leq |\langle \tilde s_{i} , \vec \ev \tilde A_1 \cdots \tilde A_i \rangle | \leq \frac12$.
Then the summability condition \eqref{eq:wm2} for a continuous eigenvalue fails.
\end{proof}

\begin{thm}\label{thm:generic}
 The set of parameters $(\alpha,\beta) \in \Omega_{\infty}$ such that
 $T_{\alpha,\beta}$ 
 %fails the summability condition \eqref{eq:wm2}
 does not have a eigenvalue $e^{2\pi i \ev}$ with $\ev \in (0,1)$
 contains a dense $G_\delta$-subset of $\Omega_{\infty}$.
\end{thm}

Hence,
%if the summability condition is indeed equivalent to the Koopman operator having a non-trivial eigenvalue, then
for a dense $G_\delta$-subset of parameters $(\alpha,\beta) \in \Omega_\infty$,
the map $G_{\alpha,\beta}$ is weak mixing.
\\[3mm]
\begin{proof}
The map $G:U \to U \cup L$ from \eqref{eq:G} is piecewise continuous. Let
$U_k = \{ (\alpha, \beta) \in \R^2 : \frac1{k+1} < \alpha \leq \frac1k, 0 \leq \beta \leq \alpha\}$, $k \in \N$, be the domains on which $G$ is continuous.
 The sets $\Omega_\infty$ and $\Delta_\infty = \bigcap_n \bigcup_{k_1, \dots, k_n \in \N^n}
 H_{k_1} \circ \cdots \circ H_{k_n}(\Delta^\circ)$ are homeomorphic
 via the coding sequences $(k_i)_{i \in \N}$ satisfying \eqref{eq:k2}.
 Indeed, every cylinder set $[k_1, \dots, k_n]$ corresponds to open triangles
 $\Omega_{k_1 \dots k_n}$ in parameter space and $\Delta_{k_1, \dots, k_n}$
 in $\Delta$, and these triangles form a topological basis of $\Omega_\infty$ and $\Delta_\infty$ respectively.
 Common boundary points of such triangles don't belong to $\Omega_\infty$ or
 $\Delta_\infty$, so $\Omega_\infty$ and
 $\Delta_\infty$ are zero-dimensional sets without isolated points, but not compact.
 The itinerary maps
 $$
\begin{cases}
  (\alpha,\beta) \mapsto (k_i)_{i \in \N}, &  G^{i-1}(\alpha,\beta) \in U_{k_i}^\circ, \\
  (u,v) \mapsto (k_i)_{i \in \N}, &  H^{1-i}(u,v) \in \Delta_{k_i}^\circ, \\
\end{cases}
$$
are homeomorphisms, and the composition $h:\Delta_\infty \to \Omega_\infty$
that assigns to $(u,v) \in \Delta_\infty$ the unique parameter pair $(\alpha,\beta) \in \Omega_\infty$ so that their itineraries coincide is the required homeomorphism
between $\Delta_\infty$ and $\Omega_\infty$.
Since $\Delta_\infty \setminus \bigcup_{p,q,r \in \Z} \ell_{p,q,r}$ is clearly a dense $G_\delta$ set and no $(u,v)$ satisfies the summability condition \eqref{eq:wm2},
the $h$-image of this set is the required dense $G_\delta$-subset of $\Omega_\infty$
of parameters failing \eqref{eq:wm2}.
\end{proof}

\subsection{Weak mixing and measurable eigenvalues for the general case.}\label{sec:wm-measurable}

For the Bratteli-Vershik representation $(V,E, \leq)$  of the ITM.
Let $h_n \in \Z^3$ be the vector with entries
$$
h_n(v) = \#\{ \text{paths from } v_0 \text{ to } v \in V_n\} = (1, \dots, 1)M_1 \cdots M_n.
$$
A necessary and sufficient condition for $e^{2\pi i \ev}$
to be a (measure-theoretic) eigenvalue is the following (\cite{BDM05} and \cite[Prop 6.122]{B23}):
There is a sequence of functions $\rho_n:\tilde V_{n+1} \to \R$ such that
\begin{equation}\label{eq:meas_ev}
g_n(x) := \left( \tilde S_n(x) + \rho_n(\ft(x_{n+1})) \right) \ev \bmod 1
  \text{ converges for $\mu$-a.e. } x \in X_{BV}
  \text{ as } n \to \infty,
\end{equation}
where $\tilde S_n(x) = \sum_{j=1}^n \langle \tilde s_j(x), h_j(\fs(x_{j+1})) \rangle$ is the minimal number $k \geq 1$ of iterates of the Vershik map such that $\tau^k(x)_{n+2} \neq x_{n+2}$.
(Recall $\tilde s_j(x)$ from \eqref{eq:wm2} with $\tilde A_k$ instead of $M_k$.)
Furthermore, $\mu$ is a $\tau$-invariant and ergodic probability measure.
The condition $\liminf_j k_j < \infty$ made in this section
implies that there is only one such measure, see Corollary~\ref{cor:EU}.

Let $\Sigma$ be the unit simplex in $\R^3$ and for $w \in \R_{\geq 0}^3 \setminus \{ 0 \}$,
let $\pi(w) = \frac{w}{ \|w \|_1}$ denote the projection of $w$ onto $\Sigma$, that is, the frequency vector $(f_v(w))_{v=1}^3$
of symbols $v$ in $w$.
In terms of properties of the matrices
$A_{k_j}$, unique ergodicity is equivalent to
$$
\pi\left( \bigcap_{j > m} A_{k_{m+1}} \cdots A_{k_j}(\R_{\geq_0}^3) \right) = \ell_m 
$$
is a single point in $\Sigma$. 
This enables us, for any sequence $(\eps_n)_{n \in \N}$,  to find a sequence $r_n \nearrow \infty$ and telescope the sequence of substitutions
(and associated matrices) accordingly, such that for $\tilde \chi_{n+1} := \chi_{k_{r_{n-1}+1}} \circ \cdots \circ \chi_{k_{r_{n+1}}}$
with associated matrices $\tilde A_{n+1} = A_{k_{r_n+1}} \cdots A_{k_{r_{n+1}}}$ we have for all $n \in \N$:
\begin{equation}\label{eq:tele1}
\sup_{v \in \{ 1,2,3\}} \sup_{w,w' \in \{1,2,3\} } \frac{ f_v(w') }{f_v(w)} \leq 
\sup_{v \in \{ 1,2,3\}} \sup_{x,y \in \R^3_{\geq 0}} \frac{ \pi(\tilde A_{n+1} x)_v }{ \pi(\tilde A_{n+1} y)_v} 
\leq 1+\eps_{n+1},
\end{equation}
for the frequencies
$f_v(w) := \frac{|\tilde \chi_{n+1}(w)|_v}{|\tilde \chi_{n+1}(w)| }$.

\begin{thm}\label{thm:wm_liminfk}
 If $\liminf_n k_n  < \infty$ and $\vec \ev$ does not belong to the stable space
 $W^s(\vec 0) \bmod 1$ for $\ev \in (0,1)$, then the corresponding ITM is weakly mixing.
\end{thm}

Contrary to Theorem~\ref{thm:wm} about the absence of continuous eigenvalues, the proof below does depend explicitly
on the substitutions $\chi_{k_i}$ rather than only its abelianizations.
It remains an open question whether our family of ITMs contains maps with a measurable non-continuous eigenvalue.

Before proving this main result of this section,
first some lemmas.

\begin{lemma}\label{lem:prefix}
For all values of $k_1, \dots, k_n$, the word $\chi_{k_1} \circ \dots \circ \chi_{k_n}(3)$ is a prefix of
 $\chi_{k_1} \circ \dots \circ \chi_{k_n}(2)$.
 If $k_n \geq 2$, then $\chi_{k_1} \circ \dots \circ \chi_{k_n}(1)$ is a prefix of
 $\chi_{k_1} \circ \dots \chi_{k_n}(3)$, up to its last letter which is $1$ or $2$
 when $n$ is even or odd.
 If $k_n = 1$ and $n \geq 2$, then $\chi_{k_1} \circ \dots \circ \chi_{k_n}(3)$ is a prefix of
 $\chi_{k_1} \circ \dots \chi_{k_n}(1)$.
\end{lemma}

\begin{proof}
Direct by induction on $n$.
\end{proof}

\begin{lemma}\label{lem:h}
 If $2 \leq k_n \leq C$ or $k_n=1, k_{n-1} \geq 2$, then
 for every $j$ such that $A_{k_{n-j}} \cdots A_{k_{n}}$ is a full matrix,
 $$
 |\chi_{k_{n-j}} \circ \dots \circ \chi_{k_n}(v)|_u \leq 4C |\chi_{k_{n-j}} \circ \dots \circ \chi_{k_n}(v')|_u
 $$
 for all $v,v' \in V_n$ and $u \in \{1,2,3\}$.
In particular, $h_n(v) \leq 4C h_n(v')$ for all $v,v' \in V_n$.
\end{lemma}

The proof is deferred to Section~\ref{sec:prooflemh}.

\begin{corollary}\label{cor:EU}
 If \eqref{eq:k2} holds and $\liminf_j k_j < \infty$, then the corresponding ITM is uniquely ergodic.
\end{corollary}

\begin{proof}
 We telescope the sequence $\chi_{k_j}$ into blocks $\chi_{ k_{n_{i-1}+1} } \circ \cdots \circ \chi_{k_{n_i}}$ such that $\tilde{A_{i}} = A_{k_{n_{i-1}+1} } \cdots A_{k_{n_i}}$ is a full matrix and $k_{n_i} \leq C$ (and if $k_{n_i} = 1$ then $k_{n_{i-1}} \geq 2$) for all $i$.
 Using Lemma~\ref{lem:h} we estimate 
  $$
 \rho(L, L') := \sqrt{
 	\frac{ \max\{ L_u/L'_u : u=1,2,3\}  }{ \min\{ L_u/L'_u : u=1,2,3\} } } \ \text{ with $L, L'$ columns of $\tilde{A_{i}}$}
 $$
 for the Hilbert metric\footnote{This time we multiply by column vector on the right, so the notation is exactly as in \cite[Formula 8.29]{B23} and transposed compared to \eqref{eq:rho}.}. For these associated matrices, we get $\rho \leq 4C$. This gives a contraction factor $\tanh(\frac12 \log(\rho)) < 1$ for the Hilbert metric, independently of $i$. Hence, unique ergodicity follows.
\end{proof}

In the sequel, write $\tilde \chi_{n} = \chi_{\tilde k_1} \circ \cdots \circ  \chi_{\tilde k_m}$ where $m = m(n)$. 
We will use Bratteli-Vershik system arguments, but still prefer to keep a metric there that agrees to the metric on the subshift. 
Hence if $(X,\tau)$ is the Bratteli-Vershik system containing distinct edge-labeled paths $x$ and $y$, then
we set $d(x,y) = 2^{-n}$ for $n = \min\{ j \geq 0 : \tau^j(x)_1 \neq \tau^j(y)_1\}$.

For $x \in X_{BV}$ and $n \in \N$ fixed, define recursively
\begin{equation}\label{eq:y}
y^0 = (y^0_m)_{m \ge 1} = x \quad \text{ and } \quad
y^{\ell+1} = \tau^{h_n(\fs(y^{\ell}_{n+1}))}(y^{\ell}).
\end{equation}

\begin{lemma}\label{lem:distance}
For every $C \geq 2$, there is $n_0 \in \N$ such that for all $n > n_0$ for which
$\tilde k_m \leq C$ and if $\tilde k_m=1$, then $\tilde k_{m-1} \geq 2$, we have
\begin{equation}\label{eq:muv}
 \mu( \{ x \in [e] : d(x, y^{\ell} ) < 2^{-n/16C}\}) \geq \frac{1}{16C} \mu([e]),
 \end{equation}
for all $\ell  \geq 0$ and $e \in E_{n+1}$, where $ [ e] = \{x \in X_{BV} : x_{n+1} = e \}$.
\end{lemma}

\begin{proof}
The telescoping is such that $\tilde \chi_{n}(v) = \chi_{\tilde k_1} \circ \cdots \circ \chi_{\tilde k_m}(v)$ 
satisfies \eqref{eq:tele1};
in particular, for each $v \in V_{n}$, $\tilde \chi_{n}(v)$ contains every symbol in $V_{n-1}$
multiple times, and $m  = m(n)$ is so large that there are
a minimal even $j' \geq 2$ and a minimal odd $j'' \geq 2$ both such that $\tilde k_{m-j'}, \tilde k_{m-j''} \geq 2$.
Due to \eqref{eq:k2} such $j'$ and $j''$ can always be found.
Set $j = \max\{ j', j''\}$.
Then $\chi_{\tilde k_{m-j}} \circ \cdots \circ \chi_{\tilde k_m}(v)$, $v \in V_n$, have a maximal common prefix $W'$ that contains every symbol.
In the light of Lemma~\ref{lem:prefix}, $W'$ equals $\chi_{\tilde k_{m-j}} \circ \cdots \circ \chi_{\tilde k_m}(1)$ minus its last letter if  $\tilde k_m \geq 2$,
and $W' =  \chi_{\tilde k_{m-j}} \circ \cdots \circ \chi_{\tilde k_m}(3)$ if $\tilde k_{m-1} > \tilde k_m = 1$.

Clearly
$$
W :=  \chi_{\tilde k_1} \circ \cdots \circ \chi_{\tilde k_{m-j-1}}(W')
$$
is a common prefix of $\tilde \chi_{n}(v)$, $v \in V_{n}$,
and it contains every symbol as well.
By Lemma~\ref{lem:prefix} we have
$$
q := \left\lfloor\ \frac12 |\tilde \chi_1 \circ \cdots \circ \tilde \chi_{n-1}(W)|\ \right\rfloor \geq \frac{1}{16 C} \max_{v \in V_{n}} |\tilde \chi_1 \circ \cdots \circ \tilde \chi_{n}(v)|.
$$

Then as all paths from $v_0$ to $\ft(e)$ in the Bratteli diagram have the same mass, the union of the first $q$ such paths in $[e]$ has mass $\geq \frac{1}{16C} \mu([e])$.

Take $v = \fs(e)$ and let $x_{\min}(v)$ be the minimal path from $v_0$ to $v$. For any
$x \in [e] \cap \tau^j([x_{\min}(v)])$ with $0 \leq j \leq q$, if follows that
$\tau^{h_n(v)}(x) \in \tau^j([x_{\min}(v')])$, where $v' \in V_{n}$ is the source of the successor edge to $x_{n+1}$ (or if $x_{n+1}$ is the maximal incoming edge, $v'\in V_n$ is the source of $\tau^{h_n(v)}(x)_{n+1}$).
Because $q$ is only half of the length of the common prefix $W$,
$d(x,y^1) = d(x, \tau^{h_n(v)}(x)) \leq 2^{-n/16C}$.
Since this is true for all  $x \in \tau^j([x_{\min}(v)])$ and $0 \leq j \leq q$, and for every $\ell \geq 0$, 
$y^{\ell} \in \tau^j([x_{\min}(v'')])$ for some $v'' \in V_{n}$, the lemma follows.
\end{proof}

\begin{proof}[Proof of Theorem~\ref{thm:wm_liminfk}]
Let $C := 3+\liminf_j k_j$.
We can assume that the telescoping of the BV-diagram can be done in such a way that \eqref{eq:tele1} holds,
and the last matrix in each telescoped block either has subscript $2 \leq k \leq C$ or is equal to $1$, but then the subscript
of the penultimate matrix is $\geq 2$.
Let $\vec \ev_0 = \vec \ev$ and $\vec \zeta_{n+1} = \vec \ev_n \cdot \tilde A_{n+1}$
and $\vec \ev_{n+1} =  \vec \zeta_{n+1} - \vec z_{n+1} \in (-\frac12,\frac12]^3$,
where $\vec z_{n+1} \in \Z^3$ is the closest integer vector to $\vec \zeta_{n+1}$.
By assumption, $\vec \ev \notin W^s(\vec 0) \bmod 1$, therefore there must be $\eta > 0$ such that
$ \|\vec \zeta_{n+1} \|_1 > \eta$ infinitely often.

Assume by contradiction that \eqref{eq:meas_ev} holds.
 Let $g(x) = \lim_{n\to\infty} g_n(x)$ wherever it converges.
 By Lusin's Theorem, for every $\delta > 0$, there is a compact subset
 $X' \subset X_{BV}$ such that $\mu(X_{BV} \setminus X') < \delta$, $g$ is uniformly continuous on $X'$
  and $\lim_{n\to\infty} g_n(x) = g(x)$ for all $x \in X'$. Pick $\delta < 1/(144C)$.
Then, for $\eps \in (0,\eta/(288C))$ arbitrary, there is $N \in \N$ such that
for every $n > N$ and $x,y \in X'$ such that $d(x,y) < 2^{-N/16C}$ we have
\begin{equation}\label{eq:cont}
 |g(x)-g(y)| < \eps \quad \text{ and } \quad |g_n(x)-g(x)| < \eps.
\end{equation}
It also follows that $|g_{n+1}(x) - g_n(x)| < 2\eps$ for all $n \geq N$.

Let $X_{n} = \{x \in \tau^j([x_{\min}(v)]) : v \in V_{n}, 0 \leq j < q\}$
with the notation $q$ and $x_{\min}(v)$ from Lemma~\ref{lem:distance}.
In fact, $X_{n}$ consists of all the paths $x \in X$ such that the finite path $(x_1, \dots, x_n)$
is among the first $q$ paths from $v_0$ to $\fs(x_{n+1})$ and $\mu(X_{n}) \geq 1/(16C)$.

Recalling also the points $y^{\ell}$ from \eqref{eq:y}, we get from Lemma~\ref{lem:distance} that 
\begin{equation}\label{eq:ee}
d(x,y^{\ell}) < 2^{-n/16C} \text{ for all } x \in X_{n} \text{ and } \ell \geq 0.
\end{equation}

Now take $x \in X_{n}$ and corresponding $y^{\ell}$ such that
$x, y^{\ell} \in X'$ and $x_{n+2} = y^j_{n+2}$ for all $0 \leq j \leq \ell$.
Then
\begin{eqnarray}\label{eq:v}
 \tb \sum_{j=0}^{\ell-1} \ev h_n(\fs(y^j_{n+1})) \tb &=&
 \tb \ev \cdot (\tilde S_n(x)-\tilde S_n(y^\ell)) \tb \nonumber \\[-4mm]
 &=& \tb \ev \cdot \tilde S_n(x) \bmod 1 - \ev \cdot  \tilde S_n(y^\ell) \bmod 1 \tb =
 \tb g_n(x) - g_n(y^\ell) \tb  \nonumber \\
&\leq& \tb g_n(x) - g(x) \tb + \tb g(x) - g(y^\ell) \tb + \tb g(y^\ell) - g_n(y^\ell) \tb < 3\eps.
\end{eqnarray}
Now take $n \geq N$ such that $\| \vec \zeta_{n+1} \| \geq \eta$.
Notice that the components of $\vec \ev_n$ are
$\ev_n(v) = \ev h_n(v) \bmod 1 \in (-\frac12, \frac12]$.
\begin{eqnarray*}
\vec \zeta_{n+1} &=& \vec \ev_n \cdot \tilde A_{n+1} = \left(\ev_n(1), \ev_n(2), \ev_n(3) \right) \cdot \tilde A_{n+1} \\
&=& \left(\sum_{j=1}^{|\tilde \chi_{n+1}(1)|}
\ev_n(\tilde\chi_{n+1}(1)_j),
\sum_{j=1}^{|\tilde \chi_{n+1}(2)|} \ev_n( \tilde \chi_{n+1}(2)_j),
\sum_{j=1}^{|\tilde \chi_{n+1}(3)|} \ev_n( \tilde \chi_{n+1}(3)_j) \right).
\end{eqnarray*}
By \eqref{eq:tele1}, for each $v \in V_n$, the frequencies
$f_v(w)$
differ among the $w \in V_{n+1}$ by no more than a factor $1+\eps_{n+1}$,
where we choose $\eps_{n+1}$ to be the smallest positive value
among the sums $\left| \sum_{v \in V_n} \ev_n(v) f_v(w) \right|$ for $w \in V_{n+1}$.
Lemma~\ref{lem:h} applied to the block of substitution
that produces $\tilde \chi_{n+1}$ gives $|\tilde \chi_{n+1}(w)| \leq 4C |\tilde \chi_{n+1}(w')|$ for all $w,w' \in V_{n+1}$.
%and also that the frequencies $f_v(w) \geq 1/4C$ for all $v \in V_n, w \in V_{n+1}$.

Since $\| \vec \zeta_{n+1}\|_1 \geq \eta$, we can pick $w \in V_{n+1}$ such that
$\left| \sum_{j=1}^{|\tilde \chi_{n+1}(w)|} \ev_n( \tilde \chi_{n+1}(w)_j) \right| \geq \eta/3$.
First assume that $\sum_{v \in V_n} \xi_n(v) f_v(w) > 0$.
Then, recalling that $\sum_{v \in V_n} |\ev_n(v)| f_v(w) \leq \frac12$,
we have
\begin{eqnarray*}
\frac{\eta}{24C}
&\leq& \frac{1}{8C} \sum_{j=1}^{|\tilde \chi_{n+1}(w)|} \ev_n( \tilde \chi_{n+1}(w)_j) \leq
\frac{1}{4C} | \tilde \chi_{n+1}(w)| \cdot \frac12
 \sum_{v \in V_n} \ev_n(v) f_v(w) \\
&\leq& \frac{1}{4C} |\tilde \chi_{n+1}(w)| \left(
 \sum_{v \in V_n} \ev_n(v) f_v(w)
 - \eps_{n+1}  \sum_{v \in V_n} |\ev_n(v)| f_v(w)  \right) \\
&\leq& \frac{1}{4C} |\tilde \chi_{n+1}(w)|  \left(
 \sum_{\ev_n(v) > 0} \ev_n(v) \frac{f_v(w)}{1+\eps_{n+1}}
 + \sum_{\ev_n(v) < 0} \ev_n(v) f_v(w)(1+\eps_{n+1}) \right)\\
 &\leq& |\tilde \chi_{n+1}(w')|  \left(
 \sum_{\ev_n(v) > 0} \ev_n(v) f_v(w')
 + \sum_{\ev_n(v) < 0} \ev_n(v) f_v(w') \right)\\
&=& |\tilde \chi_{n+1}(w')| \,  \sum_{v \in V_n} \xi_n(v) f_v(w')
= \sum_{j=1}^{|\tilde \chi_{n+1}(w')|} \ev_n( \tilde \chi_{n+1}(w')_j).
\end{eqnarray*}
If on the other hand  $\sum_{v \in V_n} \xi_n(v) f_v(w) < 0$, we change signs:
\begin{eqnarray*}
\frac{\eta}{24C}
&\leq& \frac{1}{8C}
\left| \sum_{j=1}^{|\tilde \chi_{n+1}(w)|} \ev_n( \tilde \chi_{n+1}(w)_j) \right|
\leq \left| \frac{1}{4C} | \tilde \chi_{n+1}(w)| \cdot \frac12
 \sum_{v \in V_n} \ev_n(v) f_v(w) \right| \\
&\leq& -\frac{1}{4C} |\tilde \chi_{n+1}(w)| \left(
 \sum_{v \in V_n} \ev_n(v) f_v(w)
 - \eps_{n+1}  \sum_{v \in V_n} |\ev_n(v)| f_v(w)  \right) \\
&\leq& -\frac{1}{4C} |\tilde \chi_{n+1}(w)|  \left(
 \sum_{\ev_n(v) > 0} \ev_n(v) f_v(w)(1+\eps_{n+1})
 + \sum_{\ev_n(v) < 0} \ev_n(v) \frac{f_v(w)}{1+\eps_{n+1}} \right)\\
 &\leq& -|\tilde \chi_{n+1}(w')|  \left(
 \sum_{\ev_n(v) > 0} \ev_n(v) f_v(w')
 + \sum_{\ev_n(v) < 0} \ev_n(v) f_v(w') \right)\\
&=& -|\tilde \chi_{n+1}(w')| \,  \sum_{v \in V_n} \xi_n(v) f_v(w')
= \left| \sum_{j=1}^{|\tilde \chi_{n+1}(w')|} \ev_n( \tilde \chi_{n+1}(w')_j)\right|.
\end{eqnarray*}

Recall that the $\tilde \chi_{n+1}(w)$ gives the order of incoming edges to $w \in V_{n+1}$. Without loss of generality, we can assume that \eqref{eq:tele1} applies to any subword of $\tilde \chi_{n+1}(w)$ of length at least $|\tilde \chi_{n+1}(w)|/3$.
That is, if $1 \leq a < a+|\tilde \chi_{n+1}(w)|/3 < b$, then
\begin{equation}\label{eq:ab}
\frac{\eta}{72C} \leq \left|\sum_{j=a}^{b-1}\ev_n(\tilde \chi_{n+1}(w)_j) \right|.
\end{equation}
Indeed, since the frequencies of letters in all of these subwords are almost the same,
these sums (consisting of $b-a$ terms of three types $\xi_n(1), \xi_n(2)$ and $\xi_n(3)$) indicate almost collinear vectors of lengths almost proportional to
$|b-a|/|\tilde \chi_{n+1}(w)|$.

Now choose  the most frequent (measure-wise) $w \in V_{n+1}$ and $v \in V_n$; that is
$$
\mu(\{ x \in X'_{n} : \fs(x_{n+1}) = v, \ft(x_{n+1}) = w\}) \geq \frac{1}{144C}-\delta>0.
$$
For a path $u := u_1 \dots u_n \in \tau^j([x_{min}(v)] )$ from $v_0$ to $v \in V_n$ and $0\leq j \leq q$
and for each edge $a \in E_{n+1}$ with $\fs(a) = v$ and $\ft(a) = w$, let $x^{(a)}(u)$
be a path from $v_0$ to $w \in V_{n+1}$ such that
\begin{itemize}
	\item[(i)] $x_{n+1}^{(a)}(u) = a$, %is the $a$-th incoming edge into $w$ (so $\fs(x_{n+1}^{(a)}(u)) = v$, $\ft(x_{n+1}^{(a)}(u))=w$)
	and
	\item[(ii)] $x_k^{(a)}(u) = u_k$ for all $1 \leq k \leq n$.
\end{itemize}
Hence for edges $a,b \in E_{n+1}$ connecting $v \in V_n$ and $w \in V_{n+1}$ such that \eqref{eq:ab} holds. Then
$$
\mu(X' \cap [x^{(a)}(u)]) > \frac12 \mu([x^{(a)}(u)])
\quad \text{ and }\quad
\mu(X' \cap [x^{(b)}(u)]) > \frac12 \mu([x^{(b)}(u)]). 
$$
This implies that there exists $x \in [x^{(a)}(u)] \cap X'$ and $y \in [x^{(b)}(u)] \cap X'$ such that $y = y^{\ell}$
for some $\ell \geq 1$ in the sense of \eqref{eq:y} (because $y^{(j)} \mapsto y^{(j+1)}$ is measure-preserving).

Then $\tilde S_{n-1}(x^{(a)}(u)) = \tilde S_{n-1}(x^{(b)}(u))$ and hence $g_{n-1}(x^{(a)}(u)) =  g_{n-1}(x^{(b)}(u))$.
Because $\fs(a) = \fs(b)$ and $b-a > |\tilde \chi_{n+1}(w)|/3$, it follows from \eqref{eq:ab} that
$$
\left|g_n(x^{(b)}(u)) - g_n(x^{(a)}(u))\right| =
\left|\ev \cdot \left( \tilde S_n(x^{(b)}(u)) - \tilde S_n(x^{(a)}(u))\right) \right|
= \left| \sum_{j=a}^{b-1} \ev_n(\tilde \chi_{n+1}(w)_j) \right| \geq \frac{\eta}{72C}.
$$
Therefore
\begin{eqnarray*}
4\eps &\geq& \left|(g_n(x^{(b)}(u))-g_{n-1}(x^{(b)}(u)))
+ ( g_{n-1}(x^{(b)}(u)) - g_{n-1}(x^{(a)}(u))) \right. \\
&& \left. +\ (g_{n-1}(x^{(a)}(u))-g_{n}(x^{(a)}(u)) ) \right| \\[1mm]
&=&
|g_{n}(x^{(b)}(u)) - g_{n}(x^{(a)}(u))|
%= \left| \ev ( \tilde S_n(x^{(b)}(u)) - \tilde S_n(x^{(a)}(u)) ) \right|
\geq \frac{\eta}{72C},
\end{eqnarray*}
which contradicts the choice of $\eps$. This finishes the proof.
\end{proof}

\begin{corollary}
	A linearly recurrent ITM of infinite type is weakly mixing if and only if $\vec \ev$ does not belong to the stable space
	$W^s(\vec 0)  \bmod 1$ for $\ev \in (0,1)$. Furthermore, any measurable eigenvalue is continuous.
\end{corollary}
\begin{proof}
	It follows from Theorem~\ref{thm:wm_liminfk} that if $\vec \ev \not\in W^s(\vec 0)\bmod 1$, then the system is weakly mixing. If on the other hand $\vec \ev \in W^s(\vec 0) \bmod 1$, then the convergence of
	$\| \vec \ev A_{k_1} \cdots A_{k_n}\|$
	to zero is exponential. By \cite[Theorem 1]{BDM05}, a measurable eigenvalue $e^{2\pi i \ev}$ for linearly recurrent system exists if and only if
	$$
		\sum_{n\geq1} \tb \vec\ev \tilde{A}_1 \cdots \tilde{A}_n \tb^2 < \infty
	$$
	and additionally the eigenvalue is continuous if and only if
	$$
	\sum_{n\geq1} \tb \vec\ev \tilde{A}_1 \cdots \tilde{A}_n \tb < \infty.
	$$
	Thus if $\vec \ev \in W^s(\vec 0) \bmod 1$, then the sum converges in both cases and $e^{2\pi i \ev}$ is a continuous eigenvalue of the ITM.
\end{proof}

\subsection{The proof of Lemma~\ref{lem:h}}\label{sec:prooflemh}

\begin{proof}[Proof of Lemma~\ref{lem:h}]
 Let us write $\bA^m = A_{k_m} \cdots A_{k_n}$ for the matrix associated to the substitution $\chi_{k_m} \circ \dots \circ \chi_{k_n}$.
 Then $\bA^n = A_{k_n}$ and the columns $\bA^m_2 \geq \bA^m_3$ element-wise
 for every $m \leq n$.
 
Let
$$
\bD_m := A_{k_{m-2}} \cdot A_{k_{m-1}} =
\begin{pmatrix}
 k_{m-2} & k_{m-2}-1 & k_{m-2}-1 \\ 0 & k_{m-1} & k_{m-1}-1 \\ 1 & 1 & 1
\end{pmatrix}.
$$
be the matrix associated to the next pair of substitutions $\chi_{k_{m-2}} \circ \chi_{k_{m-1}}$.
Our proof is of algorithmic nature,
illustrated by the following scheme:

\begin{figure}[h]
\begin{center}
\begin{tikzpicture}[scale=0.5]
\node at (0.6,6) {\fbox{\quad$\begin{array}{c}\text{\bf State 1} \\[1mm] \begin{pmatrix} 
 0 & k_n & k_n-1 \\ 1 & 0 & 0 \\ r & rk_n+1 & r(k_n-1)+1 \end{pmatrix}\end{array}$\quad}};
\node at (16,6) {\fbox{\quad$\begin{array}{c}\text{\bf State 2} \\[1mm] \begin{pmatrix}
  0 & k_n & k_n-1 \\ Q_1 & Q_2 & Q_3 \\ R_1 & R_2 & R_3 \end{pmatrix}\end{array}$\quad}};
 \node at (0,-1) {\fbox{\quad$\begin{array}{c}\text{\bf State 3} \\[1mm] \begin{pmatrix}
  P_1 & P_2 & P_3 \\ q & 1-q & 0 \\ R_1 & R_2 & R_3 \end{pmatrix}\end{array}$\quad}};
 \node at (15,-1) {\fbox{\quad$\begin{array}{c}\text{\bf State 4} \\[1mm]\text{$\bA^m$ is full}\end{array}$\quad}};
%%%%
\draw[->] (7.9,6) -- (10.5,6); \node at (9.2,7.75) {$k_{m-1} > 1$}; \node at (9.2,6.75) {$k_{m-2}=1$};
\draw[->] (0,3.2) -- (0,1.8);  \node at (-1.9,2.4) {$k_{m-2} > 1$};
\draw[->] (5,-1) -- (11.2,-1); \node at (7.9,-1.75) {$k_{m-1} > 1$};
\draw[->] (15,3) -- (15,1); \node at (17,2) {$k_{m-2} > 1$};
\draw[->] (7.8,3) -- (11.5,0.8); \node at (6.2,2.5) {$k_{m-1} > 1$};	\node at (8,1.5) {$k_{m-2} > 1$};

\draw[->] (-6.5,5.5) .. controls (-9.5,5) and (-9.5,7) .. (-6.5,6.5);
  \node at (-8.2,7.3) {$k_{m-2} = 1$};
   \node at (-8.2,8.2) {$k_{m-1} =$};
\draw[->] (-5,-1.5) .. controls (-8,-2) and (-8,0) .. (-5,-0.5);
    \node at (-6.5,0.2) {$k_{m-1} = 1$};
\draw[->] (21.5,5.5) .. controls (23.8,5) and (23.8,7) .. (21.5,6.5);
   \node at (22.9,7.3) {$k_{m-2} = 1$};
 \end{tikzpicture}
\caption{Every arrow stands for left multiplication with $\bD_m$}\label{fig:SFT2}
\end{center}
\end{figure}
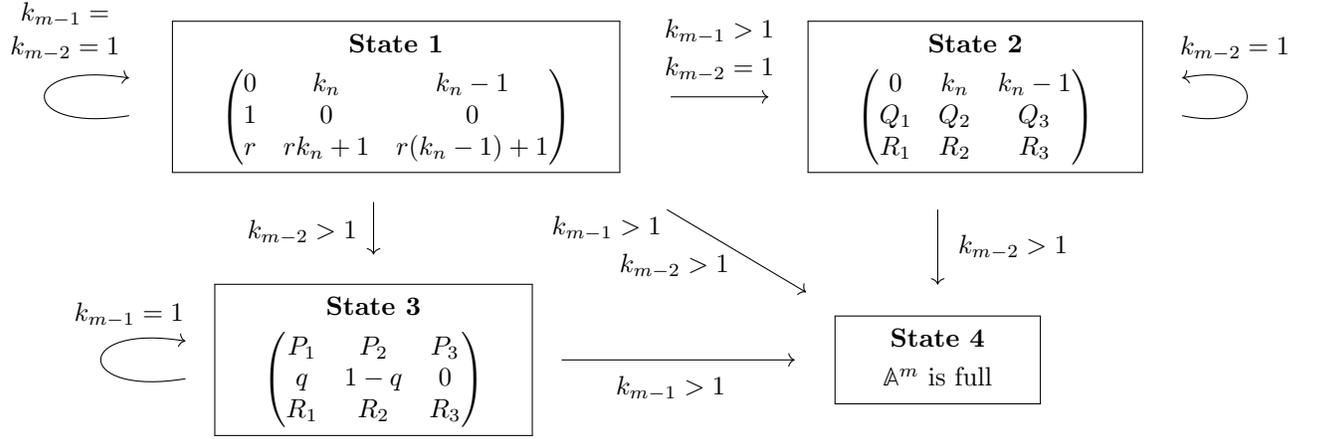
The properties of the matrices in each state are:
\begin{description}

 \item[State 1:] 
$\bA^m = \begin{pmatrix}
 0 & k_n & k_n-1 \\ 1 & 0 & 0 \\ r & rk_n+1 & r(k_n-1)+1
\end{pmatrix} 
$
for some integer $r \geq 0$.
\\
If $k_n \geq 2$, then $\bA^n$ is in this state, with $r=0$.

\item[State 2:] 
$$
\bA^m = \begin{pmatrix}
 0 & k_n & k_n-1 \\ Q_1 & Q_2 & Q_3 \\ R_1 & R_2 & R_3
\end{pmatrix} \qquad 
\begin{array}{rl} \text{for integers} &  1\leq \min_i Q_i \leq \max_j Q_j  < 3C\min_i Q_i -C \\ 
\text{ satisfying} & 1 \leq \min_i R_i \leq \max_j R_j < 3C\min_i R_i -C.
\end{array}
$$

 \item[State 3:] 
 $$
\bA^m = \begin{pmatrix}
 P_1 & P_2 & P_3 \\ q & 1-q & 0 \\ R_1 & R_2 & R_3
\end{pmatrix} \qquad 
\begin{array}{rl} \text{for integers} &  1 \leq \min_i P_i \leq \max_j P_j < 3C\min_i P_i \\ 
\text{ satisfying} 	&  1 \leq \min_i R_i \leq \max_j R_j < 3C\min_i R_i\\
					& q \in \{0,1\}.
\end{array}
$$
If $1=k_n < k_{n-1}$, then $\bA^n$ is in this state, with $q=0$.
 \item[State 4:] $\bA^m$ is full and  $(\max_i \bA^m_i)_u \leq 4C(\min_j \bA^m_j)_u$ for all $u\in\{1,2,3\}$.
 %If a repetition of State 1 happens, i.e. $r>0$, and $k_n \geq 3$ then the columns satisfy $\bA^m_1 \leq \bA^m_3 \leq \bA^m_2 \leq $C \bA^m_1$. 
 The lemma follows from these inequalities.
\end{description}

Now we verify the transitions between the states.

\begin{itemize}
\item From State 1 to State 1: $k_{m-1}  = k_{m-2} = 1$.
In this case
$$
\bD_m = \bJ := \begin{pmatrix} 1 & 0 & 0 \\ 0 & 1 & 0 \\ 1 & 1 & 1 \end{pmatrix},
$$
and we compute:
\begin{eqnarray}\label{eq:state1}
\bJ^r \cdot \bA^n = \begin{pmatrix} 1 & 0 & 0 \\ 0 & 1 & 0 \\ r & r & 1 \end{pmatrix}
 \cdot \begin{pmatrix} 0 & k_n & k_n-1 \\ 1 & 0 & 0 \\ 0 & 1 & 1 \end{pmatrix}
 = \begin{pmatrix}
 0 & k_n & k_n-1 \\ 1 & 0 & 0 \\ r & rk_n+1 & r(k_n-1)+1
\end{pmatrix},
\end{eqnarray}
as required.

 \item From State 1 to State 2: $k_{m-1} > 1 = k_{m-2}$.
 From \eqref{eq:state1} we get
\begin{eqnarray*}
\bA^{m-2} &=&
\begin{pmatrix}
 1 & 0 & 0 \\ 0 & k_{m-1} & k_{m-1}-1 \\ 1 & 1 & 1
\end{pmatrix} \cdot
 \begin{pmatrix}
 0 & k_n & k_n-1 \\ 1 & 0 & 0 \\ r & rk_n+1 & r(k_n-1) + 1
\end{pmatrix} \\[2mm]
&=&
 \begin{pmatrix}
 0 & k_n & k_n-1 \\
 X+k_{m-1} & X k_n+k_{m-1}-1 & X (k_n-1)+k_{m-1}-1\\
% (r+1)k_{m-1}-r & (k_{m-1}-1)(rk_n+1) & (k_{m-1}-1)(r(k_n-1) + 1) \\
 r+1 & (r+1)k_n+1 & (r+1)(k_n-1) + 1
\end{pmatrix}
\end{eqnarray*}
for $X = r(k_{m-1}-1) \geq 0$,
and the properties of State 2 hold.
\end{itemize}

\begin{itemize}
 \item From State 1 to State 3: $k_{m-2} > 1 = k_{m-1}$.
 From \eqref{eq:state1} we get
\begin{eqnarray*}
\bA^{m-2} &=&
\begin{pmatrix}
 k_{m-2} & k_{m-2}-1 & k_{m-2}-1 \\ 0 & 1 & 0 \\ 1 & 1 & 1
\end{pmatrix} \cdot
 \begin{pmatrix}
 0 & k_n & k_n-1 \\ 1 & 0 & 0 \\ r & rk_n+1 & r(k_n-1) + 1
\end{pmatrix} \\[2mm]
&=&
 \begin{pmatrix}
 Y-1 & Yk_n+ k_{m-2}-1 & Y(k_n-1) + k_{m-2}-1 \\
 1 & 0 & 0 \\
 r+1 & (r+1)k_n+1 & (r+1)(k_n-1) + 1
\end{pmatrix}
\end{eqnarray*}
for $Y = r(k_{m-2}-1) + k_{m-2} \geq 2$,
and the properties of State 3  hold with $q = 1$.

 \item From State 1 to State 4: $k_{m-2}, k_{m-1} > 1$.
 From \eqref{eq:state1} we get
\begin{eqnarray*}
\bA^{m-2} &=&
\begin{pmatrix}
 k_{m-2} & k_{m-2}-1 & k_{m-2}-1 \\ 0 & k_{m-1} & k_{m-1}-1 \\ 1 & 1 & 1
\end{pmatrix} \cdot
 \begin{pmatrix}
 0 & k_n & k_n-1 \\ 1 & 0 & 0 \\ r & rk_n+1 & r(k_n-1) + 1
\end{pmatrix} \\[2mm]
&=&
 \begin{pmatrix}
 Y-1 & Yk_n+ k_{m-2}-1 & Y(k_n-1) + k_{m-2}-1 \\
 X+k_{m-1} & X k_n+k_{m-1}-1 & X (k_n-1)+k_{m-1}-1\\
 r+1 & (r+1)k_n+1 & (r+1)(k_n-1) + 1
\end{pmatrix}
\end{eqnarray*}
for $X = r(k_{m-1}-1) \geq 0$, $Y = r(k_{m-2}-1) + k_{m-2} \geq 2$,
and the properties of State 4 hold.

\item From State 2 to State 2: $k_{m-1} \geq 1 = k_{m-2}$.
Left multiplication with
$$
\bD_m =
\begin{pmatrix}
 1 & 0 & 0 \\ 0 & k_{m-1} & k_{m-1}-1 \\ 1 & 1 & 1
\end{pmatrix}
$$
leaves the first row of $\bA^m$ unchanged. For the second row the inequalities for the new $\tilde{Q_i}$ follow from
\begin{align*}
\tilde{Q_i} + C&=  k_{m-1}Q_i + (k_{m-1}-1)R_i + C\\
		&< 3Ck_{m-1}Q_j + 3C (k_{m-1}-1)R_j = 3C\tilde{Q_j}
\end{align*}
and for the last row
\begin{align*}
	\tilde{R_i} + C &\leq Q_i + R_i + 2C < 3CQ_j + 3CR_j = 3C\tilde{R_j}.
\end{align*}
So the conditions of State 2 remain valid.

\item From State 3 to State 3: $k_{m-2} \geq 1 = k_{m-1}$.
Left multiplication with
$$
\bD_m = \begin{pmatrix}
 k_{m-2} & k_{m-2}-1 & k_{m-2}-1 \\ 0 & 1 & 0 \\ 1 & 1 & 1
\end{pmatrix}
$$
keeps the conditions of State 3 valid. By $R_i + 1 \leq 3CR_j$ it holds that for the new $\tilde{P}_1$
$$
	\tilde{P}_1 = P_1 + R_1 + q < 3CP_i + 3CR_i \leq 3C\tilde{P}_i
$$
for all $i$ and so on for the other entries.

\item From State 2 to State 4: $k_{m-2} > 1$.
Left multiplication with $\bD_m$ gives
\begin{align*}
 \bA^{m-2} & =  \bD_m  \cdot \bA^m =
 \begin{pmatrix}
 k_{m-2} & k_{m-2}-1 & k_{m-2}-1 \\ 0 & k_{m-1} & k_{m-1}-1 \\ 1 & 1 & 1
\end{pmatrix} \cdot 
 \begin{pmatrix}
         0 & k_n & k_n-1 \\
         Q_1 & Q_2 & Q_3 \\
         R_1 & R_2 & R_3
        \end{pmatrix}  \\[2mm]
& = \begin{pmatrix}
  (k_{m-2}-1)(Q_1+R_1)\quad & k_n k_{m-2} +  &
  (k_n-1)k_{m-2} + \\
  & (k_{m-2}-1)(Q_2 + R_2) & (k_{m-2}-1)(Q_3 + R_3) \\[2mm]
   k_{m-1} Q_1 + (k_{m-1}-1)R_1 &  k_{m-1} Q_2 + (k_{m-1}-1)R_2 &
    k_{m-1} Q_3 + (k_{m-1}-1)R_3   \\[1mm]
  Q_1+R_1 & k_n+Q_2 + R_2 & k_n-1+Q_3 + R_3
\end{pmatrix}
\end{align*}
keeps the inequality in the second and third row as before. For the first row
\begin{align*}
	\tilde{P}_2 &= k_n k_{m-2} + (k_{m-2}-1)(Q_2 + R_2) \\
				&\leq (k_{m-2}-1)(Q_2 + R_2 + C) + C \\
				&< 4C(k_{m-2}-1)(Q_j + R_j) \leq 4C\tilde{P_j}
\end{align*}
and analogously the inequalities hold for the other entries. Thus the conditions of State 4 hold.

\item From State 3 to State 4: $k_{m-2} > 1$.
Left multiplication with $\bD_m$ gives
\begin{align*}
 \bA^{m-2} & =  \bD_m  \cdot \bA^m =  \begin{pmatrix}
 k_{m-2} & k_{m-2}-1 & k_{m-2}-1 \\ 0 & k_{m-1} & k_{m-1}-1 \\ 1 & 1 & 1
\end{pmatrix} \cdot \begin{pmatrix}
         P_1 & P_2 & P_3 \\
         q & 1-q & 0\\
         R_1 & R_2 & R_3
        \end{pmatrix} = \\[2mm]
& =  \begin{pmatrix}
  k_{m-2}P_1 +   & k_{m-2}P_2 +  &  k_{m-2}P_3 +   \\
(k_{m-2}-1)(q+R_1)  & (k_{m-2}-1)(R_2 + 1 - q)  &
(k_{m-2}-1)R_3  \\[2mm]
  qk_{m-1} + (k_{m-1}-1)R_1 &  (1-q)k_{m-1} + (k_{m-1}-1)R_{2}  & (k_{m-1}-1)R_{3}\\[1mm]
  P_1+q+R_1 & P_2 + 1-q + R_2 & P_3 + R_3
\end{pmatrix}
\end{align*}
keeps the required inequalities of the first and third row as before. For the second row we see that 
\begin{align*}
	\tilde{Q}_1 &= qk_{m-1} + (k_{m-1}-1)R_1 = (k_{m-1}-1)(R_1 +q) + q \\
	&\leq 3C(k_{m-1}-1)R_j + 1 < 4C(k_{m-1}-1)R_j \leq 4C\tilde{Q_j}
\end{align*}
for all $j$ and analogously for the other entries. Hence the conditions of State 4 hold.
\end{itemize}
Since for a full matrix in State 4, 
any further left multiplications with $\bD_{m-2}$ etc., preserves the conditions of State 4, this proves the lemma for $k_n \geq 2$.
%
% For the case that $k_n = 1 < k_{n-1}$
% $$\bA^{n-1}= \begin{pmatrix}
%              k_{n-1} & k_{n-1}-1 & k_{n-1}-1 \\ 0 & 1 & 0 \\ 1 & 1 & 1
%             \end{pmatrix}
% $$
% fits in State 3 with $q = 0$ and $P_3 = P_2 < P_1 < 3CP_3$ and the same state transitions work.
\end{proof}

\end{document}